\documentclass[11pt,leqno]{article}
\usepackage{amssymb,amsfonts}%amsLatex
\usepackage{amsmath,amsthm,amsxtra}
\usepackage[all]{xy}
\usepackage{color}
\usepackage{verbatim}

\newcommand{\st}[1]{\ensuremath{^{\scriptstyle \textrm{#1}}}}

\newcommand\bigcheck[1]{#1 \raise1ex\hbox{$\hspace{-1ex}{}^\vee$}}
\newcommand\sucheck[1]{#1 \raise0.5ex\hbox{$\hspace{-1ex}{}^\vee$}}

\newcommand{\alphaparenlist}{% changes enumerate 1st level to (a)...(z)
  \renewcommand{\theenumi}{\alph{enumi}}%
  \renewcommand{\labelenumi}{(\theenumi)}%
}
\newcommand{\ad}{\mathop{\rm ad}\,}

\newcommand{\Aut}{{\rm Aut}}

\newcommand{\gl}{g\ell}

\newcommand{\Lie}{{\rm Lie}}

\newcommand{\rank}{\rm rank \, }
\newcommand{\rk}{\rm rk \, }

\renewcommand{\sl}{s\ell}

\newcommand{\ZZ}{\mathbb{Z}}
\newcommand{\FF}{\mathbb{F}}

\newcommand{\fa}{\mathfrak{a}}

\newcommand{\fg}{\mathfrak{g}}
\newcommand{\fh}{\mathfrak{h}}
\newcommand{\fk}{\mathfrak{k}}
\newcommand{\fl}{\mathfrak{l}}
\newcommand{\fm}{\mathfrak{m}}
\newcommand{\fn}{\mathfrak{n}}
\newcommand{\fo}{\mathfrak{o}}
\newcommand{\fp}{\mathfrak{p}}
\newcommand{\fq}{\mathfrak{q}}
\newcommand{\fr}{\mathfrak{r}}
\newcommand{\fs}{\mathfrak{s}}

\newcommand{\fz}{\mathfrak{z}}

\newcommand{\fsl}{\mathfrak{sl}}
\newcommand{\fsp}{\mathfrak{sp}}
\newcommand{\so}{\mathfrak{so}}

\renewcommand{\tilde}{\widetilde}

\makeatletter
\renewcommand\section{\@startsection {section}{1}{\z@}%
                                   {-3.5ex \@plus -1ex \@minus -.2ex}%
                                   {2.3ex \@plus.2ex}%
                                   {\normalfont\large\bfseries}}
\renewcommand\subsection{\@startsection{subsection}{2}{\z@}%
                                     {-3.25ex\@plus -1ex \@minus -.2ex}%
                                     {0ex \@plus .0ex}%
                                     {\normalfont\normalsize\bfseries}}

\@addtoreset{equation}{section}
\makeatother

\newtheorem{theorem}{Theorem}[section]
\newtheorem{definition}[theorem]{Definition}
\newtheorem{lemma}[theorem]{Lemma}
\newtheorem{corollary}[theorem]{Corollary}
\newtheorem{proposition}[theorem]{Proposition}

\newtheorem*{lemma*}{Lemma}

\theoremstyle{remark}
\newtheorem{remark}[theorem]{Remark}
\newtheorem{example}[theorem]{Example}

\makeatletter
\def\@maketitle{\newpage
 \null
 \vskip 2em
 \begin{center}%
  \vskip 3em
  {\Large\bf \@title \par}%
  \vskip 1.5em
  {\normalsize
   \lineskip .5em
   \begin{tabular}[t]{c}\@author
   \end{tabular}\par}%
  \vskip 2em

 \end{center}%
 \par
 \vskip 2.5em}
\makeatother

\renewcommand{\epsilon}{\varepsilon}

\definecolor{light}{gray}{.9}

\begin{document}

\title{Cyclic elements in semisimple Lie algebras}

\author{A.G. Elashvili, V.G. Kac, and E.B. Vinberg }
%Alberto De Sole\thanks{~~desole@mat.uniroma1.it~~~~Supported
% in part by Department of Mathematics, M.I.T.},~~
%Victor G. Kac
% \thanks{Department of Mathematics, M.I.T.,
% Cambridge, MA 02139, USA.~~kac@math.mit.edu~~~~Supported in part by NSF
% grants~~
% % DMS-9970007 and DMS-0201017
% }~~
% and Minoru Wakimoto\thanks{~~wakimoto@r6.dion.ne.jp~~~~Supported in part by
% Department of Mathematics, M.I.T.}
%}

\maketitle

%\pecetta{change title}
%\pecetta{add abstract}

%%%%%%%%%%%%%%%  INTRO   %%%%%%%%%%%%%%%%%%%%%
\section{Introduction}

\label{sec:intro}

Let $\fg$ be a semisimple finite-dimensional Lie algebra over
an algebraically closed field $\FF$ of characteristic 0 and let $e$ be a
non-zero nilpotent element of $\fg$.  By the
Morozov--Jacobson theorem, the element~$e$ can be included in an
$s\ell_2$-triple $\fs=\{ e,h,f \}$, so that $[e,f] =h$, $[h,e]=2e$,
$[h,f]=-2f$.  Then the eigenspace decomposition of $\fg$ with
respect to $\ad h$ is a $\ZZ$-grading of~$\fg$:
\begin{equation}
  \label{eq:0.1}
  \fg = \bigoplus^d_{j=-d} \fg_j \, ,
\end{equation}
where $\fg_{\pm d} \neq 0$.  The positive integer~$d$ is called
the {\em depth} of this $\ZZ$-grading, and of the nilpotent element $e$.
This notion was previously studied e.g. in \cite{P1}.

An element of $\fg$ of the form $e+F$, where $F$ is a non-zero
element of $\fg_{-d}$, is called a {\em cyclic element},
associated with $e$.  In \cite{K1}  Kostant proved that any cyclic
element, associated with a principal (=~regular) nilpotent
element~$e$, is regular semisimple, and in \cite{S} Springer
proved that any cyclic element, associated with a subregular
nilpotent element of a simple exceptional Lie algebra, is regular
semisimple as well, and, moreover, found two more distinguished
nilpotent conjugacy classes in $E_8$ with the same property.
Both Kostant and Springer use this property in order to exhibit
an explicit connection between these nilpotent conjugacy classes
and conjugacy classes of certain regular elements of the Weyl
group of~$\fg$.

A completely different use of cyclic elements was discovered by
Drinfeld and Sokolov \cite{DS}.  They used a cyclic element,
associated with a principal nilpotent element of a simple Lie
algebra~$\fg$, to construct a bi-Hamiltonian hierarchy of integrable
evolution PDE of $KdV$ type (the case $\fg = s\ell_2$ produces
the $KdV$ hierarchy).  In a number of subsequent papers,
\cite{W}, \cite{GHM}, \cite{BGHM}, \cite{FHM},
\cite{DF}, \cite{F},... the method of Drinfeld and Sokolov was
extended to some other nilpotent elements.  Namely, it was
established that one gets a bi-Hamiltonian integrable hierarchy
for any nilpotent element~$e$ of a simple Lie algebra, provided
that there exists a semisimple cyclic element, associated with~$e$.
One of the results of the present paper is a description of
all nilpotent elements with this property in all semisimple Lie 
algebras.

We say that a non-zero nilpotent element~$e$ (and its conjugacy class) is
of {\em nilpotent} (resp. {\em semisimple} or {\em regular
  semisimple}) {\em type} if any cyclic element, associated
with~$e$, is nilpotent (resp. any generic cyclic element,
associated with~$e$, is semisimple or regular semisimple).  If
neither of the above cases occurs, we say that $e$~is of {\em mixed
  type}.

If $\fg=\fg_1 \oplus \fg_2$ and $e=e_1+e_2$, where $e_i\in \fg_i$ are non-zero
nilpotent elements, then either $d_1=d_2$, in which case $e$ is of nilpotent
(resp. semisimple) type iff $e_i$ is such in $\fg_i$, $i=1,2$, or else we
may assume that $d_1>d_2$, in which case $e$ is of nilpotent type
iff $e_1$ is, and $e$ is of mixed type otherwise. This remark reduces the
problem of determination of types of nilpotent elements in semisimple Lie
algebras to the case when $\fg$ is simple.

By a general simple argument we prove that a nilpotent~$e$ of a simple Lie 
algebra $\fg$ is of nilpotent type if and only if its depth~$d$ is odd 
(Theorem \ref{th:1.1}).

For nilpotent elements of even depth the problem can be studied,
using the theory of theta groups \cite{Ka1}, \cite{V}, as follows.

Let $e$ be a nilpotent element of $\fg$ and (\ref{eq:0.1}) the
corresponding $\ZZ$-grading.  One can pick a Cartan subalgebra $\fh
\subset \fg_0$ of~$\fg$ and a set of simple roots $\{ \alpha_1
,\ldots ,\alpha_r \}$, such that the root subspace, attached to a
simple root~$\alpha_i$, lies in $\fg_{s_i}$ with $s_i \geq 0$,
for all $1 \leq i \leq r$.  It is a well-known result of Dynkin
\cite{D} that the only possible values of $s_i$ are $0$, $1$, or
$2$.  This associates to~$e$ a labeling by  $0$, $1$, $2$ of the 
Dynkin diagram of $\fg$, called the characteristic of~$e$, which
uniquely determines~$e$. The element $e$ is even iff all the labels 
are $0$ or $2$. If $-\alpha_0 = \sum^r_{i=1} a_i\alpha_i$ is the 
highest root of~$\fg$, then the depth $d$ of $e$ is determined by 
the formula
\begin{equation}
  \label{eq:0.2}
  d=\sum^r_{i=1} a_is_i \, .
\end{equation}

We associate to $e$ a labeled by $0$, $1$ and $2$ extended Dynkin 
diagram by putting~$s_0=2$ at the extra node and the labels $s_i$ 
of the characteristic of~$e$ at all other nodes. Let $m=d+2$, and 
let $\epsilon$ be a primitive $m$\st{th} root of $1$. Define an 
automorphism $\sigma_e$ of $\fg$ by letting
\begin{equation}
  \label{eq:0.3}
  \sigma_e (e_{\alpha_i}) = \epsilon^{s_i} e_{\alpha_i}\, , \,
  \sigma_e (e_{-\alpha_i}) =\epsilon^{-s_i} e_{-\alpha_i}\, ,
  \quad i=1,\ldots, r, \, ,
\end{equation}
where $e_{\pm \alpha_i}$ are some root vectors, attached to $\pm
\alpha_i$. The order of $\sigma_e$ is $m$ if $e$ is odd, and $m/2$
if $e$ is even.

The automorphism $\sigma_e$ defines a $\ZZ
/m \ZZ$-grading
\begin{equation}
  \label{eq:0.4}
  \fg = \bigoplus_{j \in \ZZ/m\ZZ} \fg^j \, .
\end{equation}
Here $\fg^0=\fg_0$ is a reductive subalgebra of $\fg$, whose
semisimple part has the Dynkin diagram, obtained from the extended
Dynkin diagram of $\fg$ by removing the nodes with non-zero labels.
Note also that the lowest weight vectors of the $\fg^0$-module 
$\fg^1=\fg_1$ are the roots $\alpha_i$ with $s_i=1$ \cite{Ka2}, 
\cite{OV}. If $e$ is even, then the lowest weight vectors of 
the $\fg^0$-module $\fg^2=\fg_2+\fg_{-d}$ are the roots $\alpha_i$ 
with $s_i=2$ (including the lowest root $\alpha_0$).

In the case when $e$ is even, all the labels $s_i$ are even, and
it is convenient to divide them by $2$ (which is being done in
Section 6).

The connected linear algebraic group $G^0|\fg^j$ for each 
$j\in \ZZ/m\ZZ$ is called a theta group.

One of the basic facts of the theory of theta groups is that the
$G^0$-orbit of an element $v\in \fg^j$ is closed (resp. contains
$0$ in its closure) if and only if $v$ is a semisimple
(resp. nilpotent) element of $\fg$ \cite{V}.  Since a cyclic element 
$e+F$ lies in $\fg^2$, this reduces the study of its semisimplicity 
or nilpotence to the corresponding properties of its $G^0$-orbit.

\newpage

In Section \ref{sec:2} we study the {\it rank} $\rk\, e$ of a nilpotent 
element $e$, defined as $\dim\, (\fg^2/G^0)$. It follows from Theorem 
\ref{th:1.1} that $\rk\, e>0$ if and only if the depth of $e$ is even. 
Several other equivalent definitions of the rank of $e$ are given 
by Proposition \ref{prop:2.1}. Furthermore, we prove that if $e$ is 
a nilpotent element of mixed type, then a cyclic element $e+F$ is 
never semisimple (Proposition \ref{prop:2.2}(b)). Thus, the element 
$e$ is of semisimple type if and only if there exists a semisimple 
cyclic element associated with it.
%%%%%%%%%%%%%%%%%%%%%%%%%

In Section \ref{sec:3} we introduce the notion of a {\it reducing subalgebra}
for a nilpotent element $e$ of even depth in a simple Lie algebra $\fg$.
It is a semisimple subalgebra $\fq$, normalized by $\fs$, such that
$Z(\fs)(\fq \cap \fg_{-d})$ is Zariski dense in $\fg_{-d}$. We have:
$\fs \subset \fn(\fq) =\fq\oplus \fz(\fq)$.

By a case-wise consideration, we find a reducing subalgebra $\fq$ for each
nilpotent element $e$, such that the projection of $e$ to $\fq $ along 
$\fz(\fq)$ 
has regular semisimple type in $\fq$. This leads to a characterization 
of nilpotent elements of semisimple type and to a kind of a Jordan 
decomposition of nilpotent elements. The latter decomposition allows us 
to combine all conjugacy classes of nilpotent elements of $\fg$ in 
{\it bushes}, which consist of nilpotent classes with the same reducing 
subalgebra $\fq$ and the same projection to $\fq$. Each bush contains 
a unique nilpotent class of semisimple type, and all other are of mixed 
type, but have the same depth $d$, rank $r$, and the conjugacy class
of the semisimple part of a generic cyclic element; we denote by $\fa$
the derived subalgebra of the centralizer of this semisimple part.

In Section \ref{sec:4}
we describe explicitly some minimal reducing subalgebras of nilpotent 
elements of even depth in all simple classical Lie algebras. This leads 
to a description of types of all nilpotent elements and their bushes 
in these algebras (Theorem \ref{th:4.3}).

We also find some minimal reducing subalgebras for all nilpotent elements
in the exceptional simple Lie algebras. This allows us to combine
the nilpotent classes in bushes and find their types.
All bushes of nilpotent classes in all exceptional simple Lie algebras, 
along with their invariants $d$, $r$ and $\fa$, and a minimal reducing 
subalgebra, are given in Tables 5.1-5.4 in Section 5, and in Table 1.1 
(the bush of 0).

%%%%%%%%%%%%%%%%%%%%%%%%%HERE

In Table 0.1 we list the numbers of all conjugacy classes of non-zero
nilpotent elements and those of semisimple, regular semisimple and
nilpotent type in the exceptional simple Lie algebras.

%%%%%%%%%%%%%%%%%%
\begin{table}[!h!t]\label{tab:0.1}
  \centering
  Table 0.1\\[2ex]
\begin{tabular}{c|c|c|c|c}
\hline
     \,\,\,\,\!\!\raisebox{-1ex}{$\diagdown \# $} & non-zero & nilpotent & semisimple & regular
    semisimple\\
 $\fg$ \,\,\,\,\,\raisebox{.75ex}{$\diagdown$}\,\,\,    & nilpotent classes & type classes & type classes &
            type classes\\
\hline
$E_6$ & 20 & 2 & 13 & 5\\
$E_7$ & 44 & 3 & 21 & 5\\
$E_8$ & 69 & 7 & 27 & 7 \\
$F_4$ & 15 & 2 & 11 & 4\\
$G_2$ & 4  & 1 & 3 &  2\\
\hline
  \end{tabular}
     \end{table}
In particular, among the distinguished nilpotent conjugacy classes 
we find 7~regular semisimple and 4~mixed type classes in $E_8$,
3~regular semisimple and 3~mixed type classes in $E_7$,
only regular semisimple type classes in $E_6$, $F_4$ and $G_2$
(3, 4 and 2 classes respectively).

In the case when a nilpotent  element $e$ is of regular semisimple type, i.e.
the corresponding generic cyclic element $e+F$ is regular (by a
result of Springer \cite{S} this holds for all distinguished
nilpotents of semisimple type), the centralizer of $e+F$ is a
Cartan subalgebra $\fh'$ of $\fg$, and $\sigma_e$ induces an element
$w_e$ of the Weyl group $W=W(\fh')$ of $\fg$.  We thus get a map
from the set of conjugacy classes of nilpotent elements of $\fg$
of regular semisimple type to the set of conjugacy classes of
regular elements of $W$, extending that of Kostant and Springer
\cite{K1}, \cite{S}.

As an application, we find in Section \ref{sec:6} the diagrams of all 
regular elements in exceptional Weyl groups. Another application is 
a classification of all theta groups which have a closed orbit with 
a finite stabiliser (Proposition \ref{prop:6.11}). In the cases 
$\fg=E_{6,7,8}$ this list was found in \cite{G} as a part of a long 
classification of theta groups of positive rank.

In a similar fashion we also construct a map from the set of
conjugacy classes of all nilpotent elements of semisimple type 
(and even of ``quasi-semisimple type'', see Remark \ref{rem:7.2})
to the set of conjugacy classes of W. 
%Most likely this map coincides with 
It would be interesting to compare this map with the one, 
%restriction
%of the map from the set of all conjugacy classes of nilpotent elements of
%$\fg$ to the set of conjugacy classes of $W$
constructed in \cite{KL} and in \cite{L}.
%, to the set of nilpotent conjugacy classes of semisimple type,

Throughout the paper the base field $\FF$ is algebraically closed of
characteristic zero. Also, we use throughout the following notation:
$(.,.)$ is the Killing form on $\fg$; $G$ is a simply connected connected
algebraic group over $\FF$ with the
Lie algebra $\fg$; $A|V$ denotes an algebraic group $A$ acting linearly 
on a vector space $V$, $V/A$ denotes the categorical quotient (which 
exists if $A$ is reductive), and $\rk (A|V)=\dim\, (V/A)$ is called the 
rank of this linear algebraic group; $Z(\fa)$ (resp. $\fz(\fa)$) denotes 
the centralizer of a subset $\fa$ of $\fg$ in $G$ (resp. $\fg$), 
and $N(\fa)$ (resp. $\fn(\fa)$) denotes the normaliser of a subalgebra $\fa$. 
We denote by 
$\fp'$ the derived subalgebra of a Lie algebra $\fp$.

%%%%%%%%sec 1-4 replaced with sections from vinberg's paper2/15/12
\section{The case of odd depth}
\label{sec:1}

In this section we prove the following theorem.

\begin{theorem}
  \label{th:1.1}
A nilpotent element $e$ of a simple Lie algebra $\fg$ is of nilpotent type
if and only if its depth $d$ is odd.
\end{theorem}

Consider the $\ZZ /m\ZZ$-grading (\ref{eq:0.4}) of $\fg$, corresponding to $e$,
so that any cyclic element $e+F$
(where $F\in\fg_{-d}$) lies in $\fg^2$.

Recall that, for any $\ZZ/m\ZZ$-grading of $\fg$, the dimension of a maximal
abelian subspace of $\fg^j$, consisting of semisimple elements, is called the
rank of the $G^0$-module $\fg^j$. We denote it by $\rk (G^0|\fg^j)$.
It is equal
to the dimension of the categorical quotient $\fg^j/G^0$.
If $\gcd (j_1,m)=\gcd (j_2,m)$, then $\rk (G^0|\fg^{j_1})=\rk (G^0|\fg^{j_2})$
[V].

\begin{lemma}
  \label{lem:1.1}
For the $\ZZ/m\ZZ$-grading of $\fg$ defined above, all elements of $\fg^1$ are nilpotent.
\end{lemma}

\begin{proof}
Under our definition of the $\ZZ/m\ZZ$-grading of $\fg$, one has $\fg^0=\fg_0$
and $\fg^1=\fg_1$, and hence all elements of $\fg^1$ are nilpotent.
\end{proof}

The "if" part of Theorem~\ref{th:1.1} follows immediately from this lemma and
the preceding remark.

\begin{lemma}
  \label{lem:1.3}
Let $e\in\fg$ be a nilpotent element of even depth $d$. Then
\alphaparenlist
  \begin{enumerate}
\item $((\ad e)^d x,x)$ is a $Z(\fs)$-invariant non-degenerate quadratic
form on $\fg_{-d}$.

\item The $G^0$-invariant polynomial $f(u+x)=((\ad u)^d x,x)$
on $\fg^2=\fg_2\oplus\fg_{-d}$ is non-zero.
 \end{enumerate}
\end{lemma}

\begin{proof}
The subspaces $\fg_{-d}$ and $\fg_{d}$ are dual with respect to the
Killing form, and the operator $(ade)^d$ induces an isomorphism from
$\fg_{-d}$ to $\fg_{d}$. If $d$ is even, this operator is symmetric,
whence (a) follows. Clearly, this implies (b).
\end{proof}

The "only if" part of Theorem \ref{th:1.1} follows Lemma \ref{lem:1.3}(b),
since $e+F$ is not nilpotent if $f(e+F)\neq 0$.

Nilpotent elements of nilpotent type in classical Lie algebras are described
in Section \ref{sec:4}. We list in Table 1.1 below all conjugacy classes of
nilpotent elements $e$
of nilpotent type in exceptional Lie algebras $\fg$. In the last column we
give the type of a generic (nilpotent) cyclic element $e+F$,
associated with~$e$.

%%%%%%%%%%INSERT TABLE 1.1
\begin{table}[!h]%[t]%[!h]
  \centering
Table 1.1. Nilpotent orbits of nilpotent type in exceptional Lie algebras
\\[1ex]
 \label{tab:1.1}
  \begin{tabular}{c|c|c|c}
    \hline
 $\fg$&   $e$  & $d$ & $e+F$\\
\hline  %\vspace{.75ex}
    $E_6$       & $3A_1$       & $3$     & $D_4$\\
    $E_6$       & $2A_2 +A_1$   & $5$     & $E_6$\\
    $E_7$       & $[3A_1]' $    &  $3$    & $D_4$\\
    $E_7$       & $4A_1$       &$3$      & $D_4 + A_1$\\
    $E_7$       & $2A_2 + A_1$  & $5$      & $E_6$\\
    $E_8$       &  $3A_1$      &$3$      & $D_4$\\
    $E_8$    & $ 4A_1$             & $3$     & $D_4+A_1$\\
    $E_8$    & $2A_2 + A_1$          & $5$     & $E_6$\\
    $E_8$    & $2A_2 + 2A_1$         & $5$    & $E_6 +A_1$\\
    $E_8$    & $2A_3$               & $7$    & $D_7$\\
    $E_8$    & $A_4 + A_3$           & $9$    & $E_8$\\
    $E_8$    &  $ A_7$                & $15$     & $E_8$\\
     $F_4 $   & $A_1 + \tilde{A}_1$   & $3$      & $ C_3$\\
    $F_4$     & $\tilde{A}_2 + A_1$    & $5$     & $F_4$\\
     $G_2$    & $\tilde{A}_1$          & $3$      & $G_2$\\
\hline
  \end{tabular}
\end{table}

\begin{remark}
  \label{rem:1.2}

Let $\fg=\bigoplus^{d'}_{j=-d'}\fg'_j$ be a $\ZZ$-grading of $\fg$ of depth
$d'$, such that $e\in \fg'_2$. The same proof as that of Theorem~\ref{th:1.1}
shows that $e$ is of nilpotent type for this grading if $d'$ is odd, and that
$e$ is not of nilpotent type for this grading if $d'=d$ and $d$ is even. (The
only difference is that the quadratic form from Lemma 1.3(a) may be degenerate
in this case but it is still non-zero, since $\fg_d\subset\fg'_d$ if $d'=d$.)
Note that from the classification of good $\ZZ$-gradings, given in [EK],
it is easy to see that the depth of all good gradings for $e$ is the same as
for the Dynkin grading (0.1) for all $\fg$, except for $\fg=C_n$, when
the depth may increase by $1$ or $2$ for $e$, corresponding to the partitions
with all parts even of multiplicity 2, and by $1$ for $e$, corresponding to
all other partitions with maximal part even of multiplicity $2$.
\end{remark}

\vspace*{1ex}
\section{The rank of a nilpotent element of even depth}
 \label{sec:2}

In this section $e$ is a non-zero nilpotent element of even depth $d$ in
a semisimple Lie algebra $\fg$ and $\fs=\{e,h,f\}$ is an $\fsl_2$-triple
containing $e$. Let (\ref{eq:0.1}) be the corresponding $\ZZ$-grading of
$\fg$, and let $\fg =\oplus_{j\in\ZZ/m\ZZ}\fg^j$ be the $\ZZ/m\ZZ$-grading
(\ref{eq:0.4}),
defined by the characteristic of~$e$ as in the Introduction.  In particular,
$\fg^0=\fg_0$ and $\fg^2=\fg_2+\fg_{-d}$.

Recall that an action of a reductive algebraic group is called stable
if its generic orbits are closed. In this case, the codimension of a generic
orbit is equal to the dimension of the categorical quotient. Clearly, $e$ is
of semisimple type if and only if the action $G^0|\fg^2$ is stable.

It is known that any orthogonal representation of a reductive algebraic
group is stable [Lun2]. By Lemma 1.3(a) the representation $Z(\fs)|\fg_{-d}$
is orthogonal and hence stable.

\begin{proposition}\label{prop:2.1}
The following numbers are equal:
\begin{list} {}{}
\item 1) $\dim(\fg^2/G^0)(=\rk (G^0|\fg^2))$;
\item 2) $\dim(\fg_{-d}/Z(\fs))$;
\item 3) the codimension of a generic orbit of the action $G^0|\fg^2$;
\item 4) the codimension of a generic orbit of the action $Z(\fs)|\fg_{-d}$.
\end{list}
\end{proposition}

We shall call each of these equal positive numbers the {\it rank}
of $e$ in $\fg$
and denote it by $\rk \,e$ or, more precisely, $\rk_{\fg} \,e$.

\begin{proof}
Since the orbit $G^0 e$ is open in $\fg^2$, the codimensions of generic
orbits for the actions $G^0|\fg^2$ and $Z(\fs)|\fg_{-d}$ coincide. But
the former is equal to $\dim\,(\fg^2/G^0)$, since the fibers of the
factorization morphism $\fg^2\to\fg^2/G^0$ consist of finitely many
orbits [V], while the latter is equal to $\dim\,(\fg_{-d}/Z(\fs))$, since
the action $Z(\fs)|\fg_{-d}$ is stable.
\end{proof}

The rank of a principal nilpotent element of a simple Lie algebra is equal
to $1$. If $\fg=\fg_1\oplus\fg_2$ and $e_1\in\fg_1,\,e_2\in\fg_2$ are
nilpotent elements of the same even depth, then $\rk$$e=\rk$$e_1+\rk$$e_2$.

Recall that for a reductive algebraic linear group $H|V$ and a semisimple
element $v\in V$ (i.e. such that its orbit is closed) a slice $S_v$ at $v$
is a plane of the form $v+N_v$, where $N_v\subset V$ is an $H_v$-invariant
complementary subspace of the tangent space of the orbit $Hv$ at $v$.
Recall that the linear group $H|V$ is stable if and only if the group
$H_v|N_v$ is stable \cite{Lun}.

\begin{proposition}
\label{prop:2.2}
\alphaparenlist
  \begin{enumerate}
\item If a cyclic element $e+F$ is semisimple, then $F$ is semisimple
for the action $Z(\fs)|\fg_{-d}$.

\item If there exist semisimple cyclic elements associated with $e$,
then a generic cyclic element is semisimple.
  \end{enumerate}
\end{proposition}

\begin{proof}
\alphaparenlist
  \begin{enumerate}
\item If $e+F$ is semisimple, then the orbit $G^0(e+F)$ is closed. But
$$
e+Z(\fs)F=G^0(e+F)\cap(e+\fg_{-d}),
$$
whence $Z(\fs)F$ is closed.

\item  Let a cyclic element $e+F$ be semisimple. Then, by (a), the orbit
$Z(\fs)F$ is closed. Let $R$ be the stabilizer of $F$ in $Z(\fs)$, and let $S$
be a slice at $F$ for the action $Z(\fs)|\fg_{-d}$. (One may assume that $S$ is
a subspace of $\fg_{-d}$ containing $F$.) Then $R$ is the stabilizer of $e+F$ in $G^0$, and $e+S$ is a slice at $e+F$ for the action $G^0|\fg^2$. Since the action
$Z(\fs)|\fg_{-d}$ is stable, the action $R|S$ is also stable, and hence the
action $G^0|\fg^2$ is stable as well, so a generic cyclic element is semisimple. \end{enumerate}
\end{proof}

\begin{corollary}
\label{cor:2.3}
If $e$ is a nilpotent element, for which there exist a semisimple
cyclic element $e+F$, then $e$ is of semisimple type.
\end{corollary}

\section{Reducing subalgebras}
\label{sec:3}

We retain the assumptions and notation of the previous section.

\begin{definition}\label{def:3.1}
A semisimple subalgebra $\fq\subset\fg$ is called a reducing subalgebra
for $e$, if it is normalized by $\fs$ and if
$\overline{Z(\fs)\fq_{-d}}=\fg_{-d}$
%or, in other words,
(i.e. a generic orbit of the action $Z(\fs)|\fg_{-d}$
intersects $\fq_{-d}$). We will denote by $Q$ the connected simply connected
algebraic group with Lie algebra $\fq$.
\end{definition}

\begin{remark}
\label{rem:3.2}
If $\fq$ is a reductive subalgebra normalized by $\fs$ satisfying the condition
$\overline{Z(\fs)\fq_{-d}}=\fg_{-d}$, then its semisimple part $\fq'=[\fq,\fq]$
is a reducing subalgebra for $e$. Indeed, it is obviously normalized by $\fs$,
while $\fq'_{-d}=\fq_{-d}$.
\end{remark}

If $\fq\subset\fg$ is a semisimple subalgebra normalized by $\fs$, then
$$
\fs\subset\fn(\fq)=\fq\oplus\fz(\fq).
$$
We denote by $e_\fq$ (resp. $\fs_\fq$) the projection of $e$ (resp. $\fs$)
to $\fq$ along $\fz(\fq)$. Clearly, $\fs_\fq$ is an $\fsl_2$-triple containing $e_\fq$. The following theorem is a convenient criterion for $\fq$ to be
reducing for $e$.

\begin{theorem}
  \label{th:3.3}
A semisimple subalgebra $\fq\subset\fg$ normalized by $\fs$ is reducing for $e$
if and only if $e_\fq$ has the same depth and rank in $\fq$ as $e$ in $\fg$.
\end{theorem}

\begin{proof}
Let $\fm$ be an $(\fq+\fs)$-invariant complementary subspace of $\fq$ in $\fg$.
Then
$$
\fz(\fs)=\fz_\fq(\fs_\fq)+\fz_\fm(\fs),
$$
\noindent and, for any $x\in\fq$, we have

\begin{equation}
  \label{eq:3.1}
[\fz(\fs),x]=[\fz_\fq(\fs_\fq),x]+[\fz_\fm(\fs),x],
\end{equation}

\noindent the second summand lying in $\fm$. The condition
$\overline{Z(\fs)\fq_{-d}}=\fg_{-d}$ means that, for a generic $x\in\fq_{-d}$,
one has

\begin{equation}
  \label{eq:3.2}
\fq_{-d}+[\fz(\fs),x]=\fg_{-d}.
\end{equation}

Due to (3.1), this is equivalent to the equality
$$
[\fz_\fm(\fs),x]=\fm_{-d}.
$$
It also follows from (3.1) that the codimension of $Z(\fs)x$ in $\fg_{-d}$
is equal to the codimension of $Z_H(\fs_\fq)x$ in $\fg_{-d}$ plus the codimension
of $[\fz_\fm(\fs),x]$ in $\fm_{-d}$. Hence, (3.1) holds if and only if
these codimensions are equal.
\end{proof}

\begin{remark}
\label{rem:3.4}
Let $\fl$ be the semisimple part of the centralizer of a Cartan subalgebra
of the centaliser of $\fs$. This is the well-known minimal semisimple Levi
subalgebra, containing $\fs$. It is interesting that the depth of $e$ in $\fl$
is always equal to $d$, provided that $d$ is even. We can check this by a
case-wise verification but we have no conceptual proof of this fact.
Hence $\fl$ is a good candidate for a reducing subalgebra for $e$. For
example, by Theorems \ref{th:1.1} and \ref{th:3.3}, $\fl$ is a reducing
subalgebra for $e$, provided that $\rk_\fg\, e =1$.
%Moreover, it follows after a case-wise
%analysis, that $\rk e$ does not drop when passing from $\fg$ to $\fl$
%(provided that the depth of $e$ is even). Thus, by Theorem \ref{th:3.4},
%a nilpotent element $e$ of even depth is of semisimple type in $\fg$
%if and only if it is of semisimple type in $\fq$.
%If $e$ is of semisimple type, then there exists a unique minmal reducing
%subalgebra $\fl_{min}$ for $e$ in $\fl$; it turns out that
%all nilpotents of mixed type, occuring in the bush of $e$ are obtained
%by adding to $e$ a nilpotent element from $\fz(\fl)$.
%However, sometimes the minimal, in order to combine nilpotents in a bush aroun%d a nilpotent element $e$ of semisimple type, one needs to take a
%subalgebra of $\fl$ as a reducing subalgebra for $e$.
\end{remark}
\begin{remark}
\label{rem:3.5}
The minimal semisimple Levi subalgebra $\fl$, contaning $e$, is a reducing
subalgebra if and only if the stabiliser in
$Z(\fs)$
of a generic element of $\fg_{-d}$ is a maximal torus of $Z(\fs)$.
Indeed, $\fl$ is the semisimple part of the centralizer of a maximal torus of
$Z(\fs)$, hence $\fl \cap \fg$ is the fixed point set of this torus on $\fg$.
\end{remark}
\begin{remark}
\label{rem:3.6}
By Lemma 1.3(b), the representation of the reductive group $Z(\fs)$ on
$\fg_{-d}$
is orthogonal. Hence, if
$\dim \fg_{-d}=1$,
then the representation of
$\fz(\fs)$
on $\fg_{-d}$ is trivial. Furthermore, if $\dim \fg_{-d}=2$,
this representation is trivial as well. Indeed, this representation cannot
contain $SO_2$, since in this case $r=1$, hence, by Remark \ref{rem:3.4},
$\fl$ is a reducing subalgebra, and, by Remark \ref{rem:3.5}, the generic
stabilizer must contain a maximal torus of $SO_2$, a contradiction.
Finally, if  $\dim \fg_{-d}=3$, the only possibilities for the linear group
$Z^0(\fs)| \fg_{-d}$ are: trivial, $SO_3$, and $T_1\subset SO_3$,
and in the last case $\fl$ is not a reducing subalgebra.
\end{remark}
The meaning of the notion of a reducing subalgebra is explained by the
following theorem.

\begin{theorem}
  \label{th:3.4}
Let $\fq\subset\fg$ be a reducing subalgebra for $e$. Then $e$ is of semisimple
type if and only if $e\in\fq$ and $e$ is of semisimple type in $\fq$.
\end{theorem}

\begin{proof}
By the definition of a reducing subalgebra, a generic cyclic element $e+F$
(where $F\in\fg_{-d}$) is $Z(\fs)$-conjugate to an element of
$e+\fq_{-d}\subset\fq\oplus\fz(\fq)$. The latter is semisimple if and only if
its projections to both $\fq$ and $\fz(\fq)$ are semisimple, which is only
possible if $e\in\fq$.
\end{proof}

\begin{theorem}
  \label{th:3.5}
Under the assumptions of Theorem \ref{th:3.4}, the projection
$e_{\fz}$ of $e$ to
$\fz(\fq)$ is a nilpotent element of depth $<d$ in the semisimple part
$\fz(\fq)'$ of $\fz(\fq)$.
\end{theorem}

\begin{proof}
Suppose that the depth of $e_\fz$ in $\fz(\fq)'$ is equal to $d$. Then
there is a non-zero semisimple element in $\fz(\fq)^2$, whence
$\rk_\fg\, e>\rk_\fq\, e_\fq$, which contradicts Theorem \ref{th:3.3}.
\end{proof}

\begin{proposition}\label{prop:3.6}
Let $\fq\subset\fg$ be a reducing subalgebra for $e$. Then
\begin{list} {}{}
\item (a)
Any semisimple subalgebra normalized by $\fs$ and containing $\fq$ is also
a reducing subalgebra for $e$.
\item (b)
Any reducing subalgebra for $e_\fq$ in $\fq$ is a reducing subalgebra for
$e$ in $\fg$.
\end{list}
\end{proposition}

\begin{proof}
The assertion (a) is obvious. The assertion (b) follows from the inclusion
$Z_H(\fs_\fq)\subset Z(\fs)$.
\end{proof}

Due to Proposition \ref{prop:3.6}(b), there is a minimal reducing subalgebra
for $e$. An open question is whether it is unique up to conjugation.

The simplest case, when the condition $\overline{Z(\fs)\fq_{-d}}=\fg_{-d}$
is satisfied, is when $\fq_{-d}=\fg_{-d}$. As the following proposition shows,
there is a unique minimal reducing subalgebra with the last property.

\begin{proposition}\label{prop:3.7}
Let $\fm\subset\fg$ be the $\fs$-submodule generated by $\fg_{-d}$ (the
isotypic component corresponding to the highest weight~$d$). Then the
subalgebra $\fq$ generated by $\fm$ is semisimple.
\end{proposition}

\begin{proof}
Let $\fn$ be the unipotent radical of $\fq$, and $\fq=\fl+\fn$ be a Levi decomposition. One may assume that the reductive subalgebra $\fl$ is
normalized by $\fs$.  Since the scalar product is non-degenerate on $\fm$,
we have $\fm\cap\fn=0$, so $\fm\subset\fl$; but then $\fq=\fl$ (and $\fn=0$).
Thus, $\fq$ is reductive. Clearly, $\fg_{-d}$ lies in the semisimple part $\fq'=[\fq,\fq]$ of $\fq$. Since $\fq'$ is normalized by $\fs$, we obtain
$\fm\subset\fq'$, and hence $\fq=\fq'$ is semisimple.
\end{proof}

Another way to get reducing subalgebras is given by the following proposition.

\begin{proposition}\label{prop:3.8}
Let $K$ be any reductive subgroup of a generic stabilizer of the action
$Z(\fs)|\fg_{-d}$. Then the semisimple part $(\fg^K)'$  of the subalgebra
$\fg^K$ of $K$-fixed elements of $\fg$ is a reducing subalgebra for $e$.
\end{proposition}

\begin{proof}
Clearly, the subalgebra $(\fg^K)'$ is normalized by $Z(\fs)$. By the Luna--
Richardson theorem [LR], any closed orbit of the action $Z(\fs)|\fg_{-d}$
intersects $\fg^K_{-d}$, and hence $(\fg^K)'$ is a reducing subalgebra for $e$.
\end{proof}

\begin{corollary}\label{cor:3.9}
There is a reducing subalgebra $\fq\subset\fg$ for $e$ such that a generic
stabilizer for the action $Z_Q(\fs_\fq)|\fq_{-d}$ lies in the center of $Q$.
\end{corollary}

\begin{proof}
It suffices to take for $K$ the whole of a generic stabilizer of the action
$Z(\fs)|\fg_{-d}$.
\end{proof}

\begin{corollary}\label{cor:3.10}
There is a reducing subalgebra $\fq\subset\fg$ for $e$ such that $e_\fq$
is an even nilpotent element of $\fq$.
\end{corollary}

\begin{proof}
If $e$ is odd, it suffices to take for $K$ the center of the connected
subgroup $S$ with $\Lie\, S=\fs$, which acts trivially on $\fg_{-d}$.
In other words, $\fq=\sum_j \fg_{2j}$.
\end{proof}

The main general result of the present paper is that, for any nilpotent
element $e$ of even depth, there is a reducing subalgebra $\fq\subset\fg$
such that $e_\fq$ is of semisimple type in $\fq$. Unfortunately, we can
do it only by presenting such a subalgebra in each case. This is done
in the next sections. But if it is known that such a subalgebra exists,
then the following theorem together with Proposition \ref{prop:3.6}(b)
shows that there also exists a reducing subalgebra $\fq$ such that
$e_\fq$ is of regular semisimple type in $\fq$.

\begin{theorem}
  \label{th:3.11}
Let $e\in\fg$ be a nilpotent element of semisimple type. Then there
exists a reducing subalgebra $\fq$ for $e$ such that $e$ is of
regular semisimple type in $\fq$.
\end{theorem}

\begin{proof}
Let $e+F\in\fg^2$ be a generic cyclic element associated with $e$. Note
that $K=Z(e+F)\cap G^0$ is the stabilizer of $F$ for the action
$Z(\fs)|\fg_{-d}$.

If $e$ is not of regular semisimple type, then $\fz(e+F)$ is a non-abelian
reductive $\ZZ/m\ZZ$-graded subalgebra of $\fg$. Hence, $\fz(e+F)^0\neq 0$
(see, e.g., [GOV], Theorem 3.3.7). By Proposition \ref{prop:3.8}
$\fq=(\fg^K)'$ is a reducing subalgebra for $e$ (containing $e$ and $F$).
Since a generic stabilizer for the action $Z_Q(\fs)|\fq_{-d}$ coincides
with the center of $Q$ (see Corollary \ref{cor:3.9}), we have
$\fz(e+F)\cap\fq^0=0$, which implies that $e$ is of regular semisimple
type in $\fq$.
\end{proof}

%%%%%%%%%%   start section cyclic12b.tex    3/29/12

\section{Cyclic elements in classical Lie algebras}  %1
\label{sec:4}

Let $V$ be a vector space of dimension $n$.  Let $G \subset
GL(V)$ be one of the classical linear groups $SL(V)$, $SO(V)$,
$Sp(V)$ so that $\fg:=\Lie G=\fsl(V)$, $\so(V)$ , $\fsp(V)$,
respectively.

In the last two cases the space $V$ is endowed
with a non-degenerate symmetric or skew-symmetric bilinear form,
called the scalar product.  We will refer to these two types of
vector spaces with a scalar product as quadratic and symplectic
spaces.  The tensor product of two quadratic or two symplectic
spaces (resp. of a quadratic space and a symplectic space) is
naturally a quadratic (resp. symplectic) space.

We use the notation of $\so^+(V)$ (resp. $\fsp^+(V)$) for the space
of symmetric linear operators in a quadratic (resp. symplectic)
vector space~$V$.

It is well-known (see e.g. [C]) that, up to an automorphism of $\fg$,
a nilpotent element $e \in \fg$ is defined by the
orders $n_1,\ldots ,n_p$ of its Jordan blocks.  We shall assume
that $n_1 \geq n_2 \geq \cdots \geq n_p (>0)$.

In the case $G=SO(V)$ (resp. $G=Sp(V)$) the sequence $(n_1,\ldots
,n_p)$ contains each even (resp. odd) number with even
multiplicity.

Let, as before, $\fs = \{ e,h,f\} \subset \fg$ be an $\fsl_2$-triple
containing $e$.  The space $V$ decomposes into a direct sum
\begin{equation}
  \label{eq:1}
  V=V_1 \oplus \cdots \oplus V_p
\end{equation}
of $\fs$-invariant subspaces of dimensions  $n_1,\ldots ,n_p$.
In the case $G = SO(V)$ (resp. $G=Sp(V)$) one may assume that the
summands of even (resp. odd) dimension are grouped in pairs of
dual isotropic subspaces so that their sums and the remaining
single summands are mutually orthogonal.

The eigenvalues of $h$ in $V_i$ are $n_i-1, n_i-3, \ldots
-(n_i-3), -(n_i-1))$.  Looking at the roots of $\fg$, one can
observe that $e$ is even if and only if $n_1,\ldots ,n_p$ are
of the same parity.  The depth $d$ of $e$ is equal to $2n_1-2$,
except for the cases $G=SO(V)$, $n_1$ odd, $n_1-n_2=1$, when
$d=2n_1-3$, and $G=SO(V)$, $n_1$ odd, $n_1-n_2=2$, when $d=2n_1-4$.
In particular, by Theorem \ref{th:1.1}, a nilpotent element $e$
of a classical Lie algebra $\fg$ is of nilpotent type (equivalently,
$\rk e =0$) if and only if $\fg=\so(V)$, $dim\,V \geq 7$,
$n_1$ is odd and $n_1-n_2=1$;
it is easy to see that the associated with this $e$ generic cyclic element
is a nilpotent element, corresponding to the partition with the first part
$3n_1-2$, the part $n_2$ having multiplicity two less than for $e$,
and all other parts unchanged.
So, we will not consider the latter case in the rest of this section.

We are now going to describe the behavior of cyclic elements associated
with nilpotent elements of the classical Lie algebras $\fg =\fsl(V)$,
$\so(V)$, $\fsp(V)$.  In particular, we shall obtain the
%a sort of
Jordan decomposition for them.

%In the case $G=SO(V)$, $n_1$ odd, $n_1-n_2=1$, when $d$ is odd,
%$e$ is of nilpotent type due to Theorem 1.1, so we will not
%consider this case in the rest of this section.

As before, an important role is played by the centralizer $Z(\fs)$
of the $\fsl_2$-triple $\fs = \{ e,h,f \}$.  It is a reductive subgroup
of $G$, leaving invariant each component of the grading (0.1) of $\fg$
and each isotypic component of the representation of $\fs$ in~$V$.

\begin{theorem}
Let $e$ be a non-zero nilpotent element of even depth $d$ in a classical
Lie algebra $\fg$. Then there exists a reducing subalgebra $\fg^s\subset\fg$
such that the projection $e^s$ of $e$ to $\fg^s$ is a nilpotent element
of regular semisimple type in $\fg^s$.
\end{theorem}

\begin{proof}
In what follows the theorem will be proved in seven possible cases:
one for $G=SL(V)$, two for $G=Sp(V)$, and four for $G=SO(V)$.
The subalgebra $\fg^s$ will be explicitly described in terms of
the decomposition (\ref{eq:1}). The property 1) of the definition
of a reducing subalgebra will be clear from this description.
Moreover, in all the cases but one $e^s$ will be a principal nilpotent
element of $\fg^s$ and thereby a nilpotent element of regular
semisimple type in $\fg^s$.

\section*{Case $G=SL(V)$.}

Let $n_1 = \cdots = n_r > n_{r+1}$ (assuming $n_{p+1}=0$).  Set
\begin{equation*}
\tilde{V} = V_1 \oplus \cdots \oplus V_r \,, \quad  \tilde{\fg}
    = \fsl (\tilde{V})\,.
\end{equation*}
Clearly, $\fg_{-d} = \tilde{\fg}_{-d}$. Replacing $\fg$ by $\tilde{\fg}$
and the linear operators
$e,h,f$ by their restrictions to $\tilde{V}$, one can reduce
the proof to the case, when $r=p$, i.e. all Jordan blocks of $e$
are of the same size.

Under this condition, the $\fs$-module $V$ can be represented
as
\begin{equation*}
  V=V_0 \otimes R \, ,
\end{equation*}
where $V_0$ is a simple $\fs$-module of dimension $n_1$, while
$R$ is a trivial $\fs$-module of dimension $p$.  In these
terms we have

\alphaparenlist
\begin{enumerate}
\item %%a
$e=e_0 \otimes 1$, where $e_0$ is a principal nilpotent element
of $ \fsl (V_0)$;

\item %%b
$h=h_0 \otimes 1$, where $h_0$ is a characteristic of $e_0$ in
$\fsl (V_0)$;

\item %%c
$\fg = \fsl (V_0) \otimes \gl (R)$;

\item %%d
$\fg_{-d} = \fsl (V_0)_{-d} \otimes \gl (R)$, where the
grading of $\fsl (V_0)$ is defined by $h_0$.

\end{enumerate}

The group $Z (\fs) $ is a finite extension of $1\otimes SL(R)$
(the intersection of $1\otimes GL(R)$ with $SL(V)$). It acts
on $\fg_{-d}$ by conjugations of the second tensor factor.
In this way, a generic element of $\gl (R)$ can be put in a
diagonal form, which means that the subalgebra
\begin{equation*}
  \fg^s = \fsl (V_1) \oplus \cdots \oplus \fsl (V_p) \subset \fg
\end{equation*}
satisfies property 2) of a reducing subalgebra.

\begin{remark}
  A cyclic element $e + F_0 \otimes A$ is semisimple if and only
  if $A$ is semisimple and non-degenerate.
\end{remark}

\section*{Case $G=Sp (V)$, $n_1$ even.}

As in the previous case, the proof reduces to the case, when all
Jordan blocks of $e$ have the same size.  Under this condition,
the symplectic $\fs$-module $V$ is represented as
\begin{equation*}
  V= V_0 \otimes R\, ,
\end{equation*}
where $V_0$ is a simple symplectic $\fs$-module of dimension
$n_1$, while $R$ is a quadratic vector space of dimension $p$, on
which $\fs$ acts trivially.  We also have

\alphaparenlist
\begin{enumerate}
\item %%a
 $e=e_0 \otimes 1$, where $e_0$ is a principal
 nilpotent element of $\fsp (V_0)$;

\item %%b
$h = h_0 \otimes 1$ where $h_0$ is a characteristic of $e_0$;

\item %%c
$\fg =\fsp (V_0)\otimes\so^+ (R)+\fsp^+(V_0)\otimes\so (R)$;

\item %%d
$\fg_{-d} =\fsp (V_0)_{-d}\otimes\so^+(R)$.
\end{enumerate}

The group $Z(\fs)\simeq O(R)$ acts on $\fg_{-d}$ by
conjugations of the second tensor factor.  In this way, a generic element of $\so^+(R)$ can  be put in a diagonal form in an
orthogonal basis.  Assuming that $V_1,\ldots,Vp$ are mutually
orthogonal, we thus see that
\begin{equation*}
  \fg^s =\fsp (V_1) \oplus\cdots\oplus\fsp (V_p)\subset\fg
\end{equation*}
satisfies property 2) of a reducing subalgebra.

\section*{Case $G=Sp (V)$, $n_1$ odd.}

The proof reduces to the case, when $n_1=\cdots=n_{2r}$ with
$2r=p$. Under this assumption,
\begin{equation*}
  V=V_0\otimes R\, ,
\end{equation*}
where $V_0$ is a simple quadratic $\fs$-module of dimension
$n_1$, while $R$ is a symplectic vector space of dimension $2r$,
on which $\fs$ acts trivially.  We have

\alphaparenlist

\begin{list}{}{}
\item (a)  $e=e_0 \otimes 1$, where $e_0$ is a principal
  nilpotent element of $\so (V_0)$;

\item (b)  $h=h_0 \otimes 1$, where $h_0$ is a characteristic of
  $e_0$;

\item (c)  $\fg =\so^+(V_0)\otimes\fsp (R)+\so (V_0)\otimes\fsp^+(R)$;

\item (d)  $\fg_{-d}=\so^+(V)_{-d}\otimes\fsp (R)$.
\end{list}

The group $Z(\fs)\simeq Sp (R)$ acts on $\fg_{-d}$ by conjugations
of the second tensor factor. In this way, a generic element
of $\fsp(R)$ can be put in a diagonal form in a symplectic
basis.  Assuming that the summands of (\ref{eq:1}) are grouped in
pairs of dual isotropic subspaces $V_1, V^*_1, \ldots , V_r, V^*_r$
so that their sums are mutually orthogonal, we come to the conclusion
that the algebra
\begin{equation*}
  \fg^s =\fsl (V_1)\oplus\cdots\oplus\fsl (V_r)
\end{equation*}
naturally embedded in $\fsp (V)$, satisfies property 2) of a reducing
subalgebra.

\section*{Case $G=SO(V)$, $n_1$ even.}

The proof goes as in the previous case, but this time both $V_0$
and $R$ are symplectic vector spaces, and
$\fg_{-d}=\fsp (V_0)\otimes\fsp^+(R)$.  Assuming that the summands
of (\ref{eq:1}) are grouped in pairs as in the previous case (but
with respect to the symmetric scalar product in $V$), the subalgebra
$\fg^s$ is defined in the same way.

\section*{Case $G=SO(V)$, $n_1$ odd, $n_1=n_2$.}

As in the previous cases, one may assume that $n_1=n_2=\cdots=n_p$.
Under this assumption,
\begin{equation*}
  V=V_0 \otimes R\, ,
\end{equation*}
where $V_0$ is a simple $\fs$-module of dimension $n_1$, while $R$ is a
quadratic vector space of dimension $p$,
on which $\fs$ acts trivially. We have

\alphaparenlist
\begin{enumerate}
\item %%a
$e=e_0 \otimes 1$, where  $e_0$ is a principal nilpotent element
of $\so (V_0)$;

\item   %%b
$h=h_0\otimes 1$, where $h_0$ is a characteristic of $e_0$;

\item   %%c
$\fg =\so^+(V_0)\otimes\so (R)+\so (V_0)\otimes\so^+(R)$;

\item   %%d
$\fg_{-d}=\so^+(V)_{-d}\otimes\so (R)$.
\end{enumerate}

The group $Z(\fs)\simeq SO(R)$ acts on $\fg_{-d}$ by
conjugations of the second tensor factor.

If $p=2r$, then a generic element of $\so (R)$ can be put in a
diagonal form in a basis, $(e_1,e'_1,e_2,e'_2, \ldots ,e_z,
e'_z)$ such that $(e_i, e'_i) =1$, while all the other pairs of
basis vectors are orthogonal.  Assuming that the summands of
(\ref{eq:1}) are grouped in pairs of dual isotropic subspaces
$V_1,V^*_1, \ldots , V_r,V^*_r$, so that their sums are mutually
orthogonal, we see that the algebra
\begin{equation*}
  \fg^s = \fsl (V_1)\oplus \cdots \oplus \fsl (V_r)
\end{equation*}
naturally embedded in $\fg$, satisfies property 2) of the definition
of a reducing subalgebra.

The case $p=2r+1$ is a bit different.  Here we have one extra
basis vector $e_p$ of $R$ with $(e_p,e_p)=1$, orthogonal to
all the other basis vectors, and, when putting an element of
$\so (R)$ in a diagonal form, the $p$\st{th} diagonal entry is always
zero. So if we assume that the first $2r$ summands of (\ref{eq:1})
are grouped in pairs as above, while $V_p$ is orthogonal to all of
them, then the subalgebra $\fg^s$ defined as above still satisfies
property 2) of the definition of a reducing subalgebra.

In both cases the projection of $e$ to $\fg^s$ is a principal
nilpotent element of $\fg^s$.

\section*{Case $G=SO(V)$, $n_1$ odd, $n_1-n_2=2$.}

Let $n_2 = \cdots = n_r > n_{r+1}$.  Set
\begin{equation*}
\tilde{V} = V_1 \oplus V_2 \oplus \cdots \oplus V_r\, , \quad
  \tilde{\fg}=\so (\tilde{V})\, .
\end{equation*}
Then $\fg_{-d} = \tilde{\fg}_{-d}$, so the proof reduces to the
case, when $r=p$.

Under this condition, the quadratic $\fs$-module $V$ is
represented as
\begin{equation*}
  V=V_1 \oplus (V_0 \otimes R)\, ,
\end{equation*}
where $V_0$ is a simple quadratic $\fs$-module of dimension
$n_2$, while $R$ is a quadratic vector space of dimension $p-1$,
on which $\fs$ acts trivially.  We have

\alphaparenlist
\begin{enumerate}
\item %%a
$e=e_1\oplus (e_0\otimes 1)$, where $e_1$ (resp. $e_0$) is a
principal nilpotent element of $\so (V_1)$ (resp. $\so (V_0)$);

\item %%b
$h=h_1 \oplus (h_0\otimes 1)$, where $h_1$ (resp. $h_0$) is a
characteristic of $e_1$ (resp. $e_0$);

\item %%c
$\fg=\so (V_1)\oplus (V_1\otimes V_0\otimes R)\oplus\so (V_0\otimes R)$,
where the second summand is understood as the space of skew-symmetric
linear operators taking $V_1$ to $V_0\otimes R$ and $V_0\otimes R$ to
$V_1$;

\item %%d
$d=2n_1-4$ and
\begin{equation*}
  \fg_{-d}=\so (V_1)_{-d}\oplus (V_1(-n_1+1)\otimes V_0(-n_1+3)\otimes R)\,,
\end{equation*}
where $V_1(\lambda)$ (resp. $V_0(\lambda)$) denotes the eigenspace
of $h_1$ (resp. $h_0$) of the eigenvalue $\lambda$.
\end{enumerate}

Note that the spaces $\so (V_1)_{-d}$, $V_1(-n_1+1)$, $V_0(-n_1+3)$
are one-dimensional.  The unity component of $Z(\fs)$ acts
trivially on the first summand of $\fg_{-d}$ and by unimodular
orthogonal transformations of $R$ on the second summand.  Any
non-isotropic vector of $R$ can be taken to a vector of a fixed
non-isotropic line by such a transformation. This means that the
subalgebra
\begin{equation*}
  \fg^s=\so (V_1\oplus V_2)
\end{equation*}
(assuming that $V_2$ is non-degenerate) satisfies property 2)
of the definition of a reducing subalgebra. The projection of $e$
to $\fg^s$ is not a principal nilpotent element of $\fg^s$, but
we shall show that it is of regular semisimple type in $\fg^s$.

Let $n_1=2m+1,\,n_2=2m-1$. Changing the notation, assume that $V$
is a $4m$-dimensional quadratic vector space with a basis
$\{e_i|i\in\ZZ/4m\ZZ\}$ such that $(e_i,e_{-i})=1$, while all the
other pairs of basis vectors are orthogonal. Let $A_i=-A_{-i-1}$
be the linear operator, taking $e_i$ to $e_{i+1}$,
$e_{-i-1}$ to $-e_{-i}$, and all the other basis vectors to $0$.
It is easy to see that it is skew-symmetric. The operator
$$
A=A_0+A_1+...+A_{2m-1}
$$
cyclically permutes the pairs $\pm e_0,\,\pm e_1,...,\,\pm e_{4k-1}$
and, hence, is a regular semisimple element of $\so (V)$.

The space $V$ decomposes into the orthogonal sum $V=V_1\oplus V_2$,
where $V_1$ (resp. $V_2$) is spanned
by $e_{3m},\,e_{3m-1},...,\,e_{m+1},\,e_m$
(resp. by $e_{m-1},\,e_{m-2},\,...,\,e_{-(m-2)},\,e_{-(m-1)}$).
Correspondingly, the operator $A$ decomposes as $A=e+F$, where
$F=A_{m-1}$ and
$$
e=(A_m+A_{m+1}+...+A_{2m-1})+(A_0+A_1+...+A_{m-2})
\in\so(V_1)\oplus\so(V_2)\subset\so(V)
$$
is a nilpotent element of type $(2m+1,2m-1)$ and of depth $d=4m-2$.

A characteristic of $e$ is the operator $h$ multiplying $e_i$ by
$2i$ for $i=m-1,\,m-2,...,\,-(m-2),\,-(m-1)$ and by $2i-4m$
for $i=3m,\,3m-1,...,\,m+1,\,m$. It follows that
$$
F=A_{m-1}\in\so (V)_{-d}.
$$
Thus, $e$ is a nilpotent element of regular semisimple type
in $\so (V)$.

\section*{Case $G=SO(V)$, $n_1$ odd, $n_1-n_2 >2$ or $p=1$.}

In this case $d=2n_1-4$ and $\fg_{-d}=\so(V_1)_{-d}$, so one can take
$\fg^s=\so (V_1)$.  The projection of $e$ to $\fg^s$ is a principal
nilpotent element of $\fg^s$.

\end{proof}

The description of $\fg^s$ given in the previous proof implies that
in all the cases but one the semisimple part $\fg^n$ of the
centralizer of $\fg^s$ is a classical linear Lie algebra of the same
type as $\fg$. The exceptional case is $G=SO(V),\, n_1=2$, when, if
the multiplicity of $n_1$ is $2r$, $\fg^n$ is $\so_{n-4r}$ plus the
sum of $r$ copies of $\sl_2$; but the projection of $e$ to $\fg^n$
still lies in $\so_{n-4r}$.

In all the cases but two the type of the projection $e^n$ of $e$
to $\fg^n$ is obtained from that of $e$ by deleting all the maximal
parts. The exceptional cases are $G=SO(V),\,n_1$ odd of multiplicity
$2r+1$, when one should only delete $2r$ maximal parts, and
$G=SO(V),\,n_1$ odd, $n_1-n_2=2$, when one should delete $n_1$ and
$n_2$.

In the following theorem, we retain the preceding notation and
the assumptions of Theorem 4.1. In particular, $(n_1,\ldots,n_p)$
is the partition corresponding to the nilpotent element $e\in\fg$.

%%%****insert here
\begin{theorem}
\label{th:4.3}
1) If $G=SL(V)$ or $Sp(V)$, then $e$ is of semisimple type
if and only if  the corresponding partition is of the form
$(n_1,\ldots,n_1,1,\ldots,1)$, and of regular semisimple type
if and only if $1$ occurs at most once.

2) If $G=SO(V)$, then $e$ is of semisimple type if and only if one
of the following three possibilities holds:

\alphaparenlist

\begin{enumerate}
\item %%a
$n_1$ has even multiplicity and all the other parts are $1$;
\item %%b
$n_1=2m+1, n_2=2m-1\,(m\geq 1)$ and all the other parts are $1$;
\item %%c
$n_1\geq 5$ and all the other parts are $1$.
\end{enumerate}

\noindent Moreover, $e$ is of regular semisimple type if and
only if $n_1$ is odd and $1$ occurs at most twice in case (a),
$p\leq 4$ in case (b), and $p\leq 2$ in case (c).
\end{theorem}

\begin{proof}
Let $\fg^s$ be the reducing subalgebra constructed in the proof
of Theorem 4.1. It follows from Proposition 3.8 that the element $e$
is of semisimple type if and only if it lies in $\fg^s$. The explicit
description of $\fg^s$ given in the proof of Theorem 4.1 in the seven
considered cases, permits to determine whether $e$ is of semisimple type,
in terms of the corresponding partition.

Assume now that $e\in\fg^s$. Then automatically $\fs\subset\fg^s$.
With one exception described below, denote by $V'$ the sum of all
one-dimensional components of the decomposition (\ref{eq:1}) and
by $V^s$ the sum of all the other components. The exceptional case
is the nilpotent element of type $(3,\,1,\ldots,1)$ in $\so(V)$,
when we set $V^s=V_1+V_2$ and $V'=V_3+\ldots+V_p$, assuming $V_2$ to
be orthogonal to all the other one-dimensional components.

In all the cases, denote by $\fg(V^s)$ (resp. $\fg(V')$) the classical
linear Lie algebra of the same type as $\fg$ acting on $V^s$ (resp. on
$V'$). Clearly, $\fg^s\subset\fg(V^s)$. Since a generic cyclic element
$e+F$ is $Z(\fs)$-conjugate to an element of $\fg^s$, one may assume
that it lies in $\fg^s$. Then its centralizer contains $\fg(V')$.
This implies that $e$ may be of regular semisimple type only if the
algebra $\fg(V')$ is commutative, which means that $\dim V'\leq 1$ in
the cases $G=SL(V)$ or $Sp(V)$ and $\dim V'\leq 2$ in the case $G=SO(V)$.

Moreover, in subcase (c), if $\dim V'=2$, then $\fg(V^s)=\so(V^s)$
does not contain regular semisimple elements of $\fg=\so(V)$, so $e$
cannot be of regular semisimple type.

Note that the eigenvalues of any semisimple cyclic element associated
with a principal nilpotent element of $\fsl_m$, are the $m$-th roots
of some non-zero number. In particular, if $m$ is even, they decompose
into pairs of opposite numbers. It follows that such an element is not
regular in $\so_{2m}$, where $\fsl_m$ is embedded in the natural way.
This implies that in subcase (a), if $n_1$ is even, a semisimple cyclic
element associated with $e$ cannot be regular in $\fg=\so(V)$.

It remains to check that in all the other cases the element $e$ is of
regular semisimple type. For $G=SL(V)$, for $G=Sp(V)$ and $n_1$ odd,
and for $G=SO(V)$ in subcase (a), this follows from the above
description of the eigenvalues of semisimple cyclic elements associated
with principal nilpotent elements of $\fsl_m$. For $G=Sp(V)$ and $n_1$
even, this follows from the analogous description of the eigenvalues
of semisimple cyclic elements associated with principal nilpotent
elements of $\fsp_m$, which are also the $m$-th roots of some non-zero
number.

For a nilpotent element of type $(2m+1,2m-1)$ in $so_{4m}$, the
eigenvalues of the associated semisimple cyclic element constructed
in the proof of Theorem 4.1 are $4m$-th roots of some number. This
implies that in subcase (b), the element $e$ is of regular semisimple
type, provided $\dim V'\leq 2$.

Finally, for a principal nilpotent element of $\so_{2m+1}$, the
eigenvalues of any associated semisimple cyclic element are $2m$-th
roots of some number and $0$. This implies that in subcase (c), the
element $e$ is of regular semisimple type,
provided $\dim V'\leq 1$.
\end{proof}

The rank of a nilpotent element $e$ of even depth in a classical Lie
algebra $\fg$ is determined as follows. In all the cases when the
projection of $e$ to $\fg^s$ is a principal nilpotent element of
$\fg^s$, the rank $\rk$$e$ is equal to the number of simple factors
of $\fg^s$. In the exceptional case, when $G=SO(V)$, $n_1$
odd, $n_1-n_2=2$, one has $\rk \,e=2$.

Note also that a bush of a semisimple type element $e$ in a classical Lie
algebra consists of all nilpotents, corresponding to (admissible) partitions,
obtained from the partition, associated with $e$, by replacing the set of
its $1$'s by arbitrary parts strictly smaller than the largest part,
and also, in case $G=SO(V)$ and $n_1$ of odd multiplicity $>1$, one of the
replacing parts can be $n_1$.

Finally, the minimal semisimple Levi subalgebra, containing a nilpotent element
$e$ of a classical Lie algebra $\fg$ is a reducing subalgebra, except for the
following cases:

(a) $\fg =\fs\fp_n$, $n_1$ even, $n_1=n_2$;

(b) $\fg = \fs\fo_n$, $n_1$ odd, $n_1-n_2=2$, and the multiplicity of $n_2$
is even.

%%%%%%%   end chapter 4 cyclic12b     3/29/12
%%   all replaced text is at end of document  2/15/12  jan
\vspace{3ex}

\section{Nilpotents of semisimple and mixed type in exceptional Lie algebras}
\label{sec:5}
\vspace{-.5ex}

For all exceptional simple Lie algebras $\fg$ we list in Tables
5.1-5.4 below all bushes of conjugacy classes of
nilpotent elements of even depth. In the first row of a bush we give the type
of the conjugacy class of the nilpotent of semisimple type, and in the rest
of the rows of the bush that of mixed type. All nilpotent conjugacy classes 
of a bush have the
same depth $d$ and rank $r$, the same derived subalgebra
$\fa$ of the centralizer of the semisimple part of a generic cyclic element
$e+F$, the same minimal reducing subalgebra $\fg^s$ and the semisimple part
of the centralizer $\fg^n$ of $\fg^s$ in $\fg$, and the same projection $e^s$ 
of $e$ on $\fg^s$ (for some $e$ in the conjugacy class).

Nilpotent conjugacy classes in
a bush differ only by their projection $e^n$ on $\fg^n$, listed in the next to
last column. All subalgebras $\fa$, $\fg^s$ and $\fg^n$ are semisimple
subalgebras of $\fg$. In almost all cases their conjugacy class is
determined by their types, listed in the tables. The only exceptions are
$\fg=E_7$, when $[A_5]'$,
$[A_3+A_1]'$,
and $[3A_1]'$ (resp. $[A_5]''$, $[A_3+A_1]''$,  and $[3A_1]''$) denote the
semisimple Levi subalgebras, contained (resp. not contained) in $A_7$,
and $\fg=G_2$ and $F_4$, when
$A_1$ and $A_2$ (resp. $\tilde{A}_1$ and $\tilde{A}_2$) denote the semisimple
Levi subalgebras, whose roots are long (resp. short).

Since $\fa$ is a semisimple Levi subalgebra, for a nilpotent $e$ of
semisimple type we have:
$$
\dim \,\fz_{\fg}{e+F}=\dim \, \fa + \rank \,\fg -\rank \,\fa.
$$
Also, such $e$ is of regular semisimple type iff $\fa=0$.

In the last column we list the unity component of the linear algebraic
group, which is the image of the action of $Z(\fs)$ on $\fg_{-d}$.
Here $n$ denotes the trivial linear group acting on the $n$-dimensional
vector space, $G_2$ and $F_4$ denote the $7$- and $26$-dimensional
representations respectively of these exceptional algebraic groups,
and $\oplus$ denotes the direct sum of linear algebraic groups.

\begin{table}[h]\label{tab:5.1}
  \centering
  Table 5.1.  Nilpotent orbits of semisimple and mixed type in $E_6$\\[1ex]
  \begin{tabular}{c|c|c|c|c|c|c|c|c}
\hline
 $e$ &  $d$ & $r$ &$\fa$ & $\fg^s$ &$e^s$ &$\fg^n$& $e^n$ &$Z(\fs)^0|\fg_{-d}$\\
\hline
    $A_1 $ & 2 & 1&$A_5$   & $A_1$  & $2$ & $A_5$ & 0  & 1\\
\hline
$2A_1$& $2$ &2& $A_3$&$2A_1$  & $(2;2)$  & $A_3$ & $0$ &$SO_7\oplus 1 $\\
\hline
$A_2$& 4& $1$ & $2A_2$ & $A_2$ & $ (3)$  &  $2A_2$   &  $0$ &1 \\
$A_2 +A_1$& &&&  &  & & $(2,1;1^3)$   & 1 \\
$A_2 +2A_1$& &&&  &  & & $(2,1;2,1)$ & $SO_3$ \\
\hline
$2A_2$ &4&2&$A_2$&$2A_2$&(3;3)&$A_2$&0 &$G_2 \oplus 1$\\
\hline
$A_3$  & 6 & 1 & $A_3$ &$A_3$&(4) &$2A_1$ & 0& $1$\\
$A_3+A_1$  &  &  & & &  &&  $(2;1^2)$ &$1$ \\
\hline
$D_4 (a_1)$  & $6$ & 2 & 0&$D_4$  & (5,3))&0&0 &$2$\\
\hline
$A_4$ &8& 1&  $A_1$ &$A_4$&(5)&$A_1$  &0 &$1$\\
$A_4 +A_1$&&& &&& &(2) &$1$\\
\hline
$D_4$ & $10$& $1$ &$ 2A_2$ & $B_3$& (7) &$A_1$ &0 &$1$\\
$D_5 (a_1)$&&& &&&&(2)&$1$\\
\hline
$A_5$ &10 & 1 & $A_2$ &$A_5$&(6)&$A_1$   &  $0$ &$1$  \\
\hline
$E_6 (a_3)$ & 10 & 2 &0   & $F_4 $ & $F_4 (a_2)$ & 0 & 0 &$2$ \\
\hline
 $D_5$ &14 & 1 &0 &  $B_4$ &  $(9)$ & 0 &0 &$1$ \\
\hline
$E_6(a_1) $ & 16 & 1 & 0 & $E_6$ &$E_6 (a_1)$ & 0 &0 &$1$   \\
\hline
$E_6$ & 22 & 1 & 0  & $F_4$ & $F_4$  &0&0 &$1$\\
\hline
  \end{tabular}
\end{table}
The procedure of computing the entries of Tables 5.1--5.4 is as follows.

First, we find the depth $d$ by Remark \ref{rem:1.2}, using the list of
characteristics from \cite{D} (see also \cite{E}).

Second, we compute the last column and deduce from it the value of the rank
$r$. By Remark \ref{rem:3.6}, the representation of $\fz(\fs)$ on
$\fg_{-d}$
is trivial if
$\dim \fg_{-d}=1$
or $2$, hence in these cases $r=\dim \fg_{-d}$. Also, if $e$ is
distinguished, then $\fz(\fs)=0$, and again $r=\dim \fg_{-d}$.
The remaining cases are few (three in $E_6$ and $F_4$, nine in $E_7$,
and eleven in $E_8$).
The reductive Lie algebras $\fz(\fs)$
were computed in \cite{E}, and it is straightforward to compute their
representations on $\fg_{-d}$; the number $r$ of generating invariants
in all these cases is well known.

Third, by Remark \ref{rem:3.5}, we now know from the last column all
nilpotent elements $e$, for which
the minimal semisimple Levi subalgebra $\fl$, containing $e$, is a
reducing subalgebra. Then, of course, $e$ is distingushed in $\fl$.
We see that there are only five nilpotent conjugacy classes in all the
exceptional Lie algebras when this is not the case: nilpotent classes of type
$A_3 + A_2$ and $A_3+A_2+A_1$ in $E_7$ and $E_8$, and nilpotent class of type
$\tilde{A}_1$ in $F_4$.
For $E_7$ there is only one minimal regular
subalgebra, different from $\fl$, containing these $e$, namely
$D_4 + 2A_1$ \cite{D}, and we check that the summand $D_4$ is a reducing
subalgebra. The same $D_4$ is also a reducing subalgebra for nilpotent class of
the same type in $E_8$. Hence all these four nilpotent classes are of mixed 
type. In the case of the nilpotent class $\tilde{A}_1$ in $F_4$ the only other
minimal regular subalgebra, containing $e$, is $A_1\oplus A_1$ \cite{D},
and it is easy check that it is a reducing subalgebra, in which $e$ is
of semisimple type. Hence, by Theorem \ref{th:3.4}, this nilpotent class 
in $F_4$ is of semisimple type.

\newpage

\begin{table}[!h]\label{tab:5.2}
  \centering
  Table 5.2.  Nilpotent orbits of semisimple and mixed type in $E_7$\\[1ex]
  \begin{tabular}{c|c|c|c|c|c|c|c|c}
\hline
 $e$ &  $d$ & $r$ &$\fa$ & $\fg^s$ &$e^s$ &$\fg^n$& $e^n$&$Z(\fs)^0|\fg_{-d}$\\
\hline
    $A_1 $ & 2 & 1& $D_6$& $A_1$   & $(2)$ & $D_6$ & 0 & 1 \\
\hline
$2A_1$& 2 &2& $D_4 \oplus A_1$ & $2A_1$  & $(2;2)$  & $D_4 \oplus
A_1$ & $0$ & $SO_9 \oplus 1$ \\
\hline
$[3A_1]''$&2 &3&$D_4$ & $[3A_1]''$ &(2;2;2)  &$D_4$ & 0 & $F_4 \oplus 1$\\
\hline
$A_2$ & 4 & 1 &$[A_5]''$& $A_2$ & (3) & $[A_5]''$& 0 & 1\\
$A_2 +A_1$&  &  &&  &  &&  $(2,1^4)$  & 1 \\

$A_2 +2A_1$&  &  &&  &  &&  $(2^2,1^2)$ & $SO_3$\\

$A_2+3A_1$ & & &&&&&$(2^3)$ & $G_2$\\

\hline

$2A_2$  & 4 & 2 &$A_2$ & $2A_2$ &(3;3) &$A_2$ & 0 & $G_2 \oplus SO_3$\\

\hline

$A_3$  & 6 &1  & $D_4 \oplus A_1$& $A_3$& (4)  &$[A_3\oplus A_1]''$
&0 &1\\
$[A_3+A_1]'$ &&& &&&& $(2,1^2;1^2) $& 1\\
$[A_3 +A_1]''$ &&&& &&& $(1^4;2)$ & 1\\
$A_3 +2A_1$ &&&&&&& $(2,1^2;2)$ & 1\\

\hline

$D_4 (a_1)$  & 6 & 2 & $[3A_1]''$&  $D_4$  & (5,3)& $[3A_1]''$ &0
& 2\\
%\hline
%
$D_4 (a_1)+A_1$ &&&&&&& $(2;1^2;1^2)$ & 2\\

$A_3 + A_2 $ &&&&&&& $(2;2;1^2)$ & $T_1 \subset SO_3$\\

$A_3 + A_2 +A_1$ &&&&&&& (2;2;2) & $S^2 SO_3/1$\\
\hline
$A_4$ & 8 & 1 &$A_2$& $A_4$ & (5)& $A_2$ &0 & 1\\
$A_4+A_1$ &&&&&&& (2,1) & 1\\
$A_4 + A_2$ &&&&&&& (3) & $SO_3$\\
\hline
$D_4$ & 10 & 1 & $[A_5]''$& $D_4$ & (7,1) & $[3A_1]''$ &0 & 1\\
$D_4 + A_1$ &&&&&&& $(2;1^2;1^2)$ &1\\
$D_5 (a_1)$ &&&&&&& $(2;2;1^2)$ & 1 \\
$D_5 (a_1) + A_1 $ &&&&&&& (2;2;2) & 1\\
\hline
$[A_5]'$ & 10 & 1 & $A_2 \oplus 3A_1$& $[A_5]'$ & (6) & $A_1$ &0 &1\\
\hline
$[A_5]''$ & 10 & 1 & $D_4$& $[A_5]''$ &(6) & $A_2$ &0 &1\\
$A_5 + A_1$ &&&&&&& (2,1) & 1\\
 \hline
$D_6 (a_2)$ & 10 & 2 & $A_2$& $D_6$ & (7,5) & $A_1$ &0 &2 \\
\hline
$E_6 (a_3)$ & 10 & 2 &$[3A_1]''$  & $F_4$ & $ F_4 (a_2)$ & $A_1$
&0 &2\\
 \hline
$E_7 (a_5)$ & 10 & 3 &0& $E_7$ & $E_7 (a_5)$ &0&0 &3\\
\hline
$A_6$ & 12 & 1 &0& $A_6$ & (7) &0&0 & $SO_3$\\
\hline
$D_5$ & 14 & 1  &$[3A_1]''$ & $B_4$ & (9) & $2A_1$ & 0 & 1\\
$D_5 +A_1$ &&&&&&& $(2;1^2)$ & 1 \\
$D_6 (a_1)$ &&&&&&& $(1^2;2)$ & 1  \\
$E_7 (a_4)$ &&&&&&& (2;2) & 1\\
\hline
$E_6 (a_1)$ & 16 & 1 &0& $E_6$ &$ E_6 (a_1)$ &0&0  & 1\\
\hline
$D_6$ & 18 & 1 &$A_2$& $D_6$ & (11,1) & $A_1$ &0 & 1 \\
$E_7 (a_3)$ &&&&&&& (2) & 1\\
\hline
$E_6 $ & 22 & 1  &$[3A_1]''$ & $F_4$ & $F_4$ & $A_1$& 0  & 1\\
$E_7 (a_2)$ &&&&&&& (2)& 1\\
\hline
$E_7 (a_1)$ & 26 & 1 &0& $E_7$ & $E_7 (a_1)$ &0&0 & 1\\
\hline
$E_7$ & 34 & 1 &0& $E_7$& $E_7$ &0&0 & 1\\
 \hline
  \end{tabular}
\end{table}

\newpage

\begin{table}[!h]\label{tab:5.3}
  \centering
  Table 5.3.  Nilpotent orbits of semisimple and mixed type in $E_8$\\[1ex]
  \begin{tabular}{c|c|c|c|c|c|c|c|c}
\hline
 $e$ &  $d$ & $r$ &$\fa$ & $\fg^s$ &$e^s$ &$\fg^n$& $e^n$&$Z(\fs)^0|\fg_{-d}$ \\
\hline
$A_1$ & 2 & 1  &$E_7$& $A_1$ & 2 & $E_7$ & 0 &1\\
\hline
$2A_1$ &2 & 2 &$D_6$& $2A_1$ & (2;2) & $D_6$ & 0 & $SO_{13} \oplus 1$\\
\hline
$A_2$ & 4 & 1 &$E_6$& $A_2$ & (3) & $E_6$ &0 &1\\
$A_2 + A_1$ &&&&&&& $A_1$  & 1\\
$A_2 + 2A_1$ &&&&&&& $2A_1$  & $SO_3$\\
$A_2 + 3A_1$ &&&&&&& $3A_1$   & $G_2$\\
\hline
$2 A_2$ & 4 & 2 &$2A_2$& $2A_2$ & (3;3)& $2A_2$ &0  & $G_2 \oplus  G_2$\\
\hline
$A_3$ & 6 & 1 &$D_6$& $A_3$ & (4) & $D_5$ &0 &1\\
$A_3 + A_1$ &&&&&&& $(2^2,1^6)$ & 1\\
$A_3+2A_1$ &&&&&&& $(2^4,1^2)$ &  1\\
\hline
$D_4 (a_1)$ & 6 & 2 &$D_4$& $D_4$ & (5,3) & $D_4$ &0 &2\\
$D_4 (a_1) + A_1$ &&&&&&& $(2^2,1^4)$ &2\\
$A_3 + A_2$ &&&&&&& $(2^4)$ & $T_1 \subset SO_3$\\
$A_3 + A_2 + A_1$ &&&&&&& $(3,2^2,1)$ & $S^2 SO_3 /1$\\
$D_4 (a_1) + A_2$ &&&&&&& $(3^2,1^2)$ & $Ad(SL_3)$\\
\hline
$A_4 $ & 8 & 1 &$A_4$& $A_4$ & (5) & $A_4$&0 &1\\
$A_4 + A_1$&&&&&&& $(2,1^3)$ &1\\
$A_4 + 2A_1$ &&&&&&& $(2^2,1)$ &1\\
$A_4 + A_2$ &&&&&&& $(3,1^2)$ &$SO_3$\\
$A_4 +A_2 +A_1$ &&&&&&& (3,2) & $SO_3$\\
\hline
$D_4$ & 10 & 1  &$E_6$& $B_3$ & (7)& $B_4$ & 0 &1\\
$D_4 +A_1$ &&&&&&& $(2,2,1^5)$ &1\\
$D_5 (a_1)$ &&&&&&& $(3,1^6)$ &1\\
$D_5 (a_1)+A_1$ &&&&&&& $(3,2^2,1^2)$ &1\\
$D_4 + A_2$ &&&&&&& $(3^2,1^3)$ &1\\
$D_5 (a_1) + A_2$ &&&&&&& $(3^3)$  &1\\
\hline
$A_5 $ & 10 & 1 & $D_4 \oplus A_2$& $A_5$ & (6) & $A_2 \oplus A_1$ &0 &1\\
$A_5 +A_1$ &&&&&&& $(1^3;2)$ &1\\
\hline
$E_6 (a_3)$ & 10 & 2 &$D_4$& $F_4$& $F_4 (a_2)$ & $G_2$ &0 &2\\
$E_6 (a_3) + A_1$ &&&&&&& $A_1$ &2\\
\hline
$D_6 (a_2)$ & 10 & 2 &$2A_2$& $D_6$ & (7,5) & $2A_1$ & 0 &2\\
\hline
$E_7 (a_5)$ & 10 & 3 &$A_2$& $E_7$ & $E_7 (a_5)$ & $A_1$& 0 &3\\
\hline
$E_8 (a_7)$ & 10 & 4 &0& $E_8$ & $E_8 (a_7)$ &0 &0 & 4\\
\hline

\end{tabular}
\end{table}

\newpage

%%%%%end page 2 table5.3

\begin{table}[!h]%%\label{tab:}
  \centering
  Table 5.3.  Nilpotent orbits of semisimple and mixed type in $E_8$\\
(cont'd.)\\[1ex]
  \begin{tabular}{c|c|c|c|c|c|c|c|c}
\hline
 $e$ &  $d$ & $r$ &$\fa$ & $\fg^s$ &$e^s$ &$\fg^n$& $e^n$&$Z(\fs)^0|\fg_{-d}$ \\
\hline
$A_6$ & 12 & 1 &$A_1$& $A_6$ & (7) & $A_1$ &0 & $SO_3$\\
$A_6 + A_1$ &&&&&&& (2) & $SO_3$\\
\hline
$D_5$ & 14 & 1 &$D_4$& $B_4$ & (9) & $B_3$ &0 &1\\
$D_5 + A_1$ &&&&&&& $(2^2,1^3)$  &  1  \\
$D_6 (a_1)$&&&&&&& $(3,1^4)$  &  1 \\
$D_5 +A_2$ &&&&&&& $(3^2,1)$  &  1 \\
$D_7 (a_2)$ &&&&&&& $(5,1^2)$  &  1 \\
$E_7 (a_4)$ &&&&&&& $(3,2^2)$  &  1 \\
\hline
$E_6 (a_1)$ & 16 & 1 &$A_2$& $E_6$ & $ E_6 (a_1)$ & $A_2$ & 0  &  1  \\
$E_6 (a_1) + A_1$ &&&&&&& (2,1)  &  1 \\
$E_8 (b_6)$ &&&&&&& (3)  &  1 \\
\hline
$D_6$ & 18 & 1 & $A_4$& $B_5$ & (11) & $B_2$ & 0  &  1 \\
$E_7 (a_3)$ &&&&&&& $(2^2,1)$  &  1 \\
$D_7 (a_1)$ &&&&&&& $(3,1^2)$  &  1 \\
\hline
$E_8 (a_6)$ &18 & 2 &0 & $E_8$ & $E_8 (a_6)$ &0&0 & 2\\
\hline
$E_6 $ & 22 & 1 & $D_4$& $F_4$ & $F_4$ & $G_2$ & 0  &  1  \\
$E_6 + A_1$ &&&&&&& $A_1$  &  1 \\
$E_7 (a_2)$ &&&&&&& $\tilde{A}_1$  &  1 \\
$E_8 (b_5)$ &&&&&&& $G_2 (a_1)$  &  1 \\
\hline
$D_7$ & 22 & 1 &$2A_2$& $B_6$ & (13)& $A_1$ & 0  &  1 \\
\hline
$E_8 (a_5)$ & 22 & 2 &0& $E_8$ & $E_8 (a_5)$ &0&0  & 2\\%%%
\hline
$E_7 (a_1)$ & 26 & 1 &$A_1$& $E_7$ & $E_7 (a_1)$ & $A_1$ &0  &  1 \\
$E_8 (b_4) $ &&&&&&& (2)  &  1 \\
\hline
$E_8 (a_4)$ & 28 & 1 &0&  $ E_8$ & $E_8 (a_4)$ &0&0  &  1 \\
\hline
$E_7$ & 34 & 1 &$A_2$& $E_7$ & $E_7$ & $A_1$ & 0  &  1 \\
$E_8 (a_3)$ &&&&&&& (2)  &  1 \\
\hline
$E_8 (a_2)$ & 38 & 1 &0& $E_8$ & $E_8 (a_2)$ &0&0  &  1 \\
\hline
$E_8 (a_1)$ & 46 & 1 &0& $E_8$ & $E_8 (a_1)$ &0&0  &  1 \\
\hline
$E_8$ & 58 & 1 &0& $E_8$ & $E_8$ & 0 & 0  &  1 \\
\hline
  \end{tabular}
\end{table}

\newpage

\begin{table}[!h]\label{tab:5.4}
  \centering
  Table 5.4.  Nilpotent orbits of semisimple and mixed type in $F_4$
  and $G_2$\\[1ex]
  \begin{tabular}{c|c|c|c|c|c|c|c|c|c}
\hline
$\fg$ & $e$ &  $d$ & $r$ &$\fa$ & $\fg^s$ &$e^s$ &$\fg^n$& $e^n$&$Z(\fs)^0|\fg_{-d}$\\

\hline
 $G_2$&   $A_1 $ & 2 & 1&  $\tilde{A}_1$& $A_1$   & $(2)$ & $\tilde{A}_1$ & 0&1  \\
\hline
        & $G_2 (a_1)$ & 4& 1 & 0 & $A_2$ & (3) &0&0 &1\\
\cline{2-10}%%%\hline
 & $G_2$ & 10 & 1 & 0 & $G_2$ & $G_2$ & 0 & 0 &1\\
\hline  %%\cline{2-9}
$F_4$ & $A_1$ & 2 & 1 & $C_3$ & $ A_1$ & (2) & $C_3$ &0 &1\\
\hline
   & $\tilde{A}_1$ & 2 & 2 & $B_2$ & $2A_1$ & (2;2) & $2A_1$ & 0
   &$SO_6\oplus 1$\\
\cline{2-10} %%\hline
   & $A_2$ & 4 & 1 & $\tilde{A}_2$ & $A_2$& (3) & $\tilde{A}_2$
   & 0 &1\\
   & $A_2 + \tilde{A}_1$ & &&&&&& (2,1)& $SO_3$\\
\cline{2-10}  % \hline
 & $\tilde{A}_2$ & 4 & 1 & $A_2$ & $\tilde{A}_2$& (3) & $A_2$  &
 0 & $G_2$\\
\cline{2-10}  % \hline
   & $B_2$ & 6 & 1 & $B_2$ & $B_2$ & (5) & $2A_1$ & 0 &1\\
 & $C_3 (a_1)$ &&&&&&& $(2;1^2)$  &1\\
\cline{2-10}   %\hline
     & $F_4 (a_3)$ & 6 & 2 & 0 & $D_4$ & (5,3) & 0 & 0  &2\\
\cline{2-10}  %%% \hline
    & $B_3$ & 10 & 1 & $\tilde{A}_2$ & $B_3$ & (7) &0 &0 & 1\\
\cline{2-10}  %%%\hline
    &  $C_3$ & 10 & 1 & $ A_2$ & $C_3$ &(6) & $A_1$ & 0 &1 \\
\cline{2-10}  %%%\hline
   & $F_4 (a_2)$ & 10 & 2 & 0 & $F_4$ & $F_4 (a_2)$ &0&0 &2\\

\cline{2-10}  %%%\hline
   & $F_4 (a_1)$ & 14 & 1 & 0 & $B_4$ & (9) &0&0 &1\\
\cline{2-10}  %%%\hline
    & $F_4$ & 22 & 1 & 0 & $F_4$ & $F_4$ & 0 &0 &1\\
\hline

  \end{tabular}
\end{table}

Now we turn to the nilpotent classes $e$, for which the subalgebra $\fl$ is a
reducing subalgebra. Since $e$ is distinguished in $\fl$, and, by Theorem
\ref{th:3.4}, $e$ is of semisimple type in $\fg$ if and only if it is
in $\fl$, we need to figure out which of the distingushed nilpotent classes
are of semisimple type. For this we use the following lemma.
\begin{lemma}
\label{lem:5.1}
(a) Let $e$ be a distingushed (hence even) nilpotent element of semisimple type
in a simple Lie algebra $\fg$. As described in the introduction, we associate
with $e$ an automorphism $\sigma_e$ of $\fg$ of order $m=\frac{1}{2}d +1$,
where $d$ is the depth of $e$ (recall that for even $e$ we divide the
RHS of  (\ref{eq:0.2}) by 2).
Then $m$ is a regular number of $\fg$, defined in Section 6, and we have:
$|\Delta|=m \dim \fg^\sigma$, where $\Delta$ is the set of roots of $\fg$.

(b) The following is a complete list of distinguished nilpotent conjugacy 
classes of semisimple (hence of regular semisimple \cite{S}) type in all simple
Lie algebras:

(i) regular nilpotent classes;

(ii) subregular nilpotent classes $X(a_1)$ in all exceptional Lie algebras $X$;

(iii) nilpotent classes, corresponding to the partition
$(2k+1, 2k-1, 1)$ in $B_{2k}$, and to the partition $(2k+1,2k-1)$ in
$D_{2k}$ (denoted by $D_{2k}(a_{k-1})$);

(iv) nilpotent classes $F_4(a_2)$, $E_7(a_5)$, $E_8(a_i)$ for $i=2,4,5,6,7$.
\end{lemma}
\begin{proof}
Let $e$ be an even nilpotent element of $\fg$ of regular semisimple type,
and let $\sigma_e$ be the corresponding automorphism of order $m$ of
$\fg$, desribed in the introduction.
By the Kostant-Springer construction (see Section 6), the centralizer of
a generic cyclic element $e+F$, associated with $e$, is a Cartan subalgebra
$\fh'$
of $\fg$, on which $\sigma_e$ induces a regular element $w_e$ of the Weyl
group. It has only zero fixed points in $\fh'$, and $e+F$ is its regular
eigenvector whose eigenvalue is $m$-th primitive root $\epsilon$ of 1.
Hence $w_e$ is a
regular element of the Weyl group, and its order
$m$ is a regular number of $\fg$, equal to $d/2 +1$, proving (a).

The classifiation of all distinguished elements of semisimple type in classical
Lie algebras follows from Theorem \ref{th:4.3}.

All regular numbers are listed in \cite{S} (see also Section 6),
and the condition that $d/2 +1$ is a regular number rules out all
distingushed nilpotents in exceptional Lie algebras, which are not listed
in (b), with the exception of the nilpotent $E_8(b_5)$. The latter
is ruled out by the last equation in (a) since in this case
$\dim \fg^{\sigma_e}=22$.

By the results of Kostant \cite{K1} and Springer \cite{S}, the regular and
subregular nilpotents of exceptional Lie algebras are of regular semisimple
type, as well as the subsubregular nilpotent $E_8(a_2)$.

In the remaining cases, listed in (b)iv, the elements $\sigma_e$
are powers of the above, and they still have only zero fixed points
in $\fh'$ and a regular eigenvector with eigenvalue the corresponding power of
$\epsilon$. Hence all the corresponding theta groups are stable, i.e. all these
distinguished nilpotent elements are of regular semisimple type.

Here is a simple proof that the regular nilpotent element $e$
has semisimple type, using the theory of theta groups (which works over any
field). Indeed,
in this case $G^0$ is an r-dimensional torus and a generic cyclic element
is of the form $\sum_{i=0}^r e_{\alpha_i}$
in the notaton of the introduction. Since 0 lies inside the convex hull
of its weights $\alpha_i$, $i=0,...,r$, it follows that the orbit of such
element is closed. A similar proof works, for example, for the subprinciplal
nilpotent classes in exceptional Lie algebras.
\end{proof}

We thus get the list of all nilpotent classes of semisimple type in all 
exceptional
Lie algebras. For each of them we choose as a reducing subalgebra $\fg^s$ the
minimal semisimple Levi subalgebra $\fl$, containing $e$, with the following
exceptions, when we take for $\fg^s$ a smaller subalgebra:
if $e$ is of type $D_n$, $E_6$, $E_6(a_3)$,
$F_4(a_1)$, $F_4(a_3)$, $G_2(a_1)$
in $\fl$, then we take $\fg^s=B_{n-1},\, F_4,\, F_4,\, B_4,\, D_4,\, A_2$
respectively.

Next, for each nilpotent element $e^s$ of semisimple type we
compute the centarizer of a generic cyclic element; its derived subalgebra
is $\fa$. This computation is done on the computer, using the programm
by W. de Graaf \cite{G}.

After that, for each $e^s$ we combine all nilpotent classes, having the same
$d$, $r$, and the conjugacy class of $\fa$, as $e^s$, in one bush.
We compute $\fg^n = \fz(\fg^s)'$, and we check that all nilpotents of mixed
type in one bush lie in $\fg^s \oplus \fg^n$ and have projection $e^s$ on
$\fg^s$.
For computation of $\fg^n$ the following remark is useful.

\begin{remark}
Since the semisimple part of a cyclic element lies in $\fg^s$, it follows that
$\fg^n \subset \fa$. This inclusion is an equality in many cases.
\end{remark}

The type of a nilpotent element $e^s+e^n$ with $e^n\in \fg^n$ is established
by computing the dimension and the reductive part of its centralizer,
using the programm by W. de Graaf \cite{G}.
%By \cite{S} $e$ is of regular semisimple type,
%hence the centaliser $\fh'$ in $\fg$ of a generic cyclic element $e+F$
%is a Cartan subalgebra of $\fg$
\begin{remark}
In most of the cases bushes are compatible with inclusions of semisimple Lie
algebras $\fk \subset \fg$ of the same rank. Namely, if $e^s$ is a nilpotent
element of
semisimple type in $\fk$, such that its minimal semisimple Levi subalgebra in
$\fk$ is a reducing subalgebra, then $e^s$ is of semisimple type in $\fg$
as well, and the bush of $e^s$ in $\fk$ is contained in the bush of $e^s$
in $\fg$ (some of the nilpotent orbits of the former combine in one
nilpotent orbit in the latter).
\end{remark}

For example, the bush of the nilpotent class $D_4$ of semisimple type
in $E_8$ comes from this nipotent class in the subalgebra $D_8$,
where it corresonds to the partition $(7,1^9)$. The bush of it
in $D_8$ consists of 7 nilpotent orbits (as described in Section 4),
two of which get joined in a single orbit in $E_8$. As a result, we get
all 6 nilpotent orbits of the bush of the nilpotent class $D_4$ in $E_8$.

Note that for the nilpotent class of type $D_5$ in $E_7$ we have:
$\fg^n=A_1\oplus A_1$.
One of these $A_1$'s corresponds to a root of $E_7$, and its centralizer
in $E_7$ is $D_6$. The latter contains a direct sum of $B_4$ and the other
$A_1$, which does not correspond to a root of $E_7$ (it is of type $2A_1$).

In all these considerations paper \cite{La} was very useful.

\begin{remark}
Let us call a nilpotent element $e$ (and its conjugacy class) 
{\it irreducible} if there is no 
proper reducing subalgebra for it. From the discussion in Section 4 and 
Tables 5.1--5.4 we obtain the following complete list of conjugacy classes
of irreducible nilpotent
elements in all simple Lie algebras:
$A_{2n},\, B_n (n\neq 3),\, C_n,\, D_{2n}(a_{n-1}),\, G_2,\, F_4,\, F_4(a_2),\,
E_6(a_1),\, E_7, \, E_7(a_1), \,E_7(a_5), \,E_8, \,E_8(a_i)$ for 
$i=1,2,4,5,6,7$.
\end{remark}

In Tables 5.5 and 5.6 below, we list all nilpotent classes $e$ of mixed type
(combined in bushes) in the exceptional Lie algebras. For each $e$, we give 
the dimension of the centralizer $\fz(e+F)$ of a generic cyclic element $e+F$,
and the nilpotent part of $e+F$ in $\fa$ (the derived subalgebra of the 
centralizer of the semisimple part of $e+F$). For this we again used the 
programm by W. de Graaf \cite{G}.

%% \newpage

\begin{table}[!h]\label{tab:5.5}
  \centering
  Table 5.5.  Nilpotent orbits of mixed type in $E_6$, $E_7$, $F_4$,
  and $G_2$\\[1ex]
  \begin{tabular}{c|c|c|c|c}  %%%%|c|c|c|c|c}
\hline
$\fg$ & $e$ & $\dim \fz(e+F)$ &$\fa$ & nilpotent part of $e+F$ in $\fa$\\
\hline
$E_6$ & $A_2+A_1$ & 14 & $2A_2$ & $(1^3 ; 2,1)$\\
%%%\cline{2-5}
& $A_2 + 2A_1$ & 10 &&(2,1 ;2,1)\\

\cline{2-5}
 & $A_3 + A_1$ & 12 & $A_3$ & $(2,1^2)$\\
\cline{2-5}
& $A_4 + A_1$ & 6 & $A_1$ & (2)\\
\cline{2-5}
& $D_5 (a_1)$ & 10 & $2A_2$ & (2,1;2,1)\\
\hline
$E_7$ & $A_2 + A_1$ & 27 & $[A_5]''$ & $(2,1^4)$\\
 &  $A_2 + 2A_1$ & 21 && $(2^2 , 1^2)$\\
 & $A_2 + 3A_1$ & 19 && $(2^3)$\\
\cline{2-5}
&   $[A_3+A_1]''$ & 31 & $D_4 \oplus A_1$ & $(1^8 ;2)$\\
&   $ [A_3 + A_1]'$ & 23 && $(2^2 , 1^4 ; 1^2)$\\
&   $A_3 + 2A_1$ & 21 && $(2^2, 1^4; 2)$\\
\cline{2-5}
&   $D_4 (a_1 ) + A_1 $ & 11 & $[3A_1]''$ & $(1^2 ; 1^2;2)$\\
&   $A_3 + A_2$ & 9 && $(1^2;2;2)$\\
&   $A_3 + A_2 +A_1$ &7 && $(2;2;2)$\\
\cline{2-5}
 &   $D_4 +A_1$ & 27 & $[A_5]''$ & $(2,1^4)$\\
&   $D_5 (a_1)$ & 21 && $(2^2,1^2)$\\
 & $D_5 (a_1) + A_1$ & 19 && $(2^3)$\\
\cline{2-5}
 &  $A_4 +A_1$ & 9 & $A_2$ & (2,1)\\
 & $A_4 + A_2$ & 7 &  & (3)\\
\cline{2-5}
 & $A_5 + A_1$ & 21 & $D_4$ & $(2^2,1^4)$\\
\cline{2-5}
 & $D_5 + A_1$ & 11 & $[3A_1]''$ & $(1^2;1^2;2)$\\
 & $D_6 (a_1)$ & 9 && $(1^2; 2;2)$\\
& $E_7 (a_4)$ & 7 && (2;2;2)\\
\cline{2-5}
 & $E_7 (a_3)$ & 9 & $A_2$ & (2,1)\\
\cline{2-5}
 & $E_7 (a_2)$ & 7 & $[3A_1]''$ & (2;2;2)\\
\hline
$F_4$ & $A_2 + \tilde{A}_1$ & 6 & $\tilde{A}_2$ & (2,1)\\
\cline{2-5}
 & $C_3 (a_1)$ & 8 & $B_2$ & $(2^2,1)$\\
\hline

  \end{tabular}
\end{table}
\newpage

\begin{table}[!h]\label{tab:5.6}
  \centering
  Table 5.6.  Nilpotent orbits of mixed type in $E_8$\\[1ex]
  \begin{tabular}{c|c|c|c}  %%%%|c|c|c|c|c}
\hline
$e$ &   $\dim \fz(e+F)$ &$\fa$ & nilpotent part of $e+F$ in $\fa$\\
\hline
$A_2 + A_1$ &  58 &  $E_6$  & $A_1$\\
$A_2 + 2A_1 $ & 48 && $2A_1$\\
$A_2 +3A_1$ & 40 && $3A_1$\\
\hline %%%\cline{2-4}
$A_3 + A_1$ & 50 & $D_6$ & $(2^2,1^8)$\\
$A_3 + 2A_1$ & 40 && $(2^4,1^4)$\\
\hline %%%\cline{2-4}
$D_4 (a_1) + A_1$   & 22 &    $D_4$ & $(2^2,1^4)$\\
$A_3 + A_2$ & 20   &     & $(2^4)$\\
$A_3 + A_2 +A_1$    & 16 && $(3,2^2,1)$\\
$D_4 (a_1) + A_2$    & 14 && $(3^2,1^2)$\\
\hline
$A_4+A_1$        & 20      & $A_4$ & $(2, 1^3)$\\
$A_4 + 2A_1$         & 16 && $(2^2 ,1)$\\
$A_4 + A_2$          & 14 && $(3,1^2)$\\
$A_4 +A_2 +A_1$     & 12 && (3,2)\\
\hline
$D_4 + A_1$         & 58 & $E_6$  &$A_1$\\
$D_5 (a_1)$          & 48 && $2A_1$\\
$D_5 (a_1) + A_1$    & 40 && $3A_1$\\
$D_4 + A_2$           & 38 && $A_2$\\
$D_5 (a_1) + A_2$    & 30 && $2A_1 + A_2$\\
\hline
$A_5 + A_1$      & 28   & $D_4+A_2$ & $(2^2,1^4;1^3)$\\
\hline
$E_6 (a_3 ) + A_1$   & 22 & $D_4$ & $(2^2, 1^4)$\\
\hline
$A_6 + A_1$           & 8     & $A_1$ & (2)\\
\hline
$D_5 + A_1$ & 22 &$D_4$ & $(2^2,1^4)$\\
$D_6 (a_1)$ & 20 && $(2^4)$\\
$E_7 (a_4)$ & 16 && $(3,2^2,1)$\\
$D_5 + A_2$ & 14 && $(3^2,1^2)$\\
$D_7 (a_2)$ & 12 && $(4^2)$\\

\hline
$E_6 (a_1) + A_1$ & 10 & $A_2$ & (2,1)\\
$E_8 (b_6) $  &8   && (3)\\
\hline
$E_7 (a_3)$ & 20 & $A_4$ & $(2,1^3)$\\
$D_7 (a_1)$ & 16 && $(2^2,1)$\\
\hline
$E_6 + A_1$ & 22 & $D_4$ & $(2^2,1^4)$\\
$E_7 (a_2)$ & 16 && $(3,2^2,1)$\\
$E_8 (b_5)$ & 14 && $(3^2,1^2)$\\
\hline
$E_8 (b_4)$& 8 & $A_1$ & (2)\\
\hline
$E_8 (a_3)$ & 10 & $A_2$ & (2,1)\\
\hline
  \end{tabular}
\end{table}

\newpage
\vspace{-2ex}

\newpage

\vspace*{1ex}
\section{Diagrams of regular elements of exceptional Weyl groups}
\label{sec:6}

Let $\fh$ be a Cartan subalgebra of a simple Lie algebra~$\fg$
and let~$W\subset \Aut_{\FF} \fh$ be the Weyl group of~$\fg$ .
An element $w \in W$ is called {\em regular} \cite{S} if it has
an eigenvector $a \in \fh$, which is  a regular element
(i.e.,~the centralizer of ~$a$ in $\fg$ is $\fh$), and $w\neq 1$.  The order of a regular
element of~$W$ is called a {\em regular number}.  Springer proved
in \cite{S} that, up to conjugacy, there exists at most one,
up to conjugacy,
regular element in~$W$ of a given order and listed all possible
regular numbers, along with the characteristic polynomials of the
corresponding regular elements.

By definition, if~$e$ is a nilpotent element
of regular semisimple type, then the corresponding generic cyclic
element $e+F$ is regular semisimple, hence its centralizer is a
Cartan subalgebra, which we denote by~$\fh'$.  As explained in
the introduction,
% for an even~$e$ and in Section~1 for an odd~$e$, we
we associate with~$e$ a $\ZZ /m\ZZ$ grading (0.4), hence an
inner automorphism of $\fg$ of order~$m$, denoted by $\sigma_e$, such that
%
%\begin{equation*}
$\sigma_e (e) = \epsilon e$,
% \,\, \hbox{(resp. $= \epsilon^2
%    e$)\,\,} \hbox{\,\, if $e$ is even (resp. odd)}\, .
%\end{equation*}
%
where $\epsilon$ is $m$-th primitive root of 1. The order $m$ is given by
the RHS of (\ref{eq:0.2}) if $e$ is odd, and by its half if $e$ is even.

The automorphism  $\sigma_e$ leaves $\fh'$ invariant and induces on it a
regular element of the Weyl group, which we denote by~$w_e$.
Note that $\sigma_e$ and $w_e$ have the same order if $e$ is even.
It follows
from the classification of nilpotent elements of regular
semisimple type, given in Sections~4 and 5, that, with the exception of
the nilpotent of $\fs\ell_n$, corresponding to the partition $(p^m\,
, \, 1)$ (hence $n=mp+1$), where $p$ is even, all nipotent elements
of $\fg$ of regular semisimple type are even.

In the following Tables 6.1 - 6.5 we list the diagrams of all
%in the lower part the diagrams
inner automorphisms $\sigma$ of all exceptional Lie algebras $\fg$,
which induce the
regular elements $w$ of the Weyl group $W$ and have the same order $m$
as $w$. In the third column we list $\fg^{\sigma}$ and in the last one the
dimension of $\fh^w$, if it is non-zero. By Kostant's theorem \cite{K2},
there is a unique
lift of $w\in W$ if $\fh^w=0$. However there can be several conjugacy classes
of such $\sigma$, even of the same order as $w$, if $\fh^w\neq 0$.
We list in the lower parts of the tables the $\sigma$ of the form
$\sigma =\sigma_e$, where $e$ are
distinguished nilpotent elements of (regular) semisimple type
(for all of them $\fh^w$ must be 0).
In the middle part the tables we list $\sigma$ of the form $\sigma_e$,
corresponding to remaining nilpotents of regular
semisimple type (among exceptional Lie algebras they exist only
for $E_6$ and $E_7$); for each regular number they appear first in the tables.
After them we list those $\sigma$ of the same order, which induce the same $w$.
Finally, it turns out that in all
exceptional Lie algebras (in fact, in all simple Lie algebras,
except for $\fsl_n$ with $n$ odd)
%and $so_{2n}$ with $n$~odd)
all regular numbers are divisors of the orders of the elements $w_e$,
corresponding to nilpotents~$e$ of regular semisimple type.
Hence, taking the appropriate powers of $\sigma_e$, we get the
diagrams of finite order automorphisms of $\fg$ which induce
all the remaining regular elements of~$W$.  They are given
in the upper parts of the tables.
% (it happens that there is only one
%for each regular number in all these cases).
It turns out that in all cases, except for
$E_6^{(1)}$ and $E_7^{(1)}$, we have
$\fh^w =0$ for all regular $w$. In the last column of Tables 6.1 and 6.2
we list $\dim \fh^w$ when it is not~0 (in all Tables 6.3-6.5 it is ~0).

It is interesting to note that all elements of the form
$e^{\frac{\pi i\rho^\vee}{m}}$
, where $\rho^\vee$ is the half of the sum
of positive coroots and m is a regular number, are conjugate to an element
from the tables, which is the first one for given $m$. It turns out that the
conjugacy classes of the $\sigma$ represented in the tables contain all $m$-th
order inner automorphisms of $\fg$, whose fixed point set has minimal possible
dimension among all $m$-th order inner automorphisms, cf. \cite{EKV}.
Note also that the second element for given $m$ in the tables is
a conjugate of the element of the form $e^{\frac{\pi i(\rho^\vee -\omega)}{m}}$
%inner automorphisms of finite
%order~$\fg$, presented in these tables are unique, up to
%conjugacy, inner automorphisms of this order with the given
%fixed  point subalgebra of $\fg$.  In fact, all inner
%automorphisms of $\fg$ of minimal order with one of these fixed
%point subalgebras, form a single conjugacy class.
where $\omega$ is the fundamental coweight, attached to the branching
node of the Dynkin diagram. Of course in
the case when the fixed point subalgebra $\fg^{\sigma}$ is~$\fh$, we get the
well-known conjugacy class of Coxeter--Kostant elements of order
equal to the Coxeter number $h$.

%\newpage
%%%\vspace{-2ex}
%  \centerline{Diagrams of regular elements of $W$}
\begin{table}[!h]%[t]%[!h]
  \centering
Table 6.1.  Diagrams of regular elements of $W_{E_6}$\\[1ex]\label{tab:6.1}
  \begin{tabular}{c|c|c|c}
    \hline
order m&   $E^{(1)}_6$  & fixed point set & $\dim \fh^w$\\
\hline  %\vspace{.75ex}
&0&\\
&1&\\
2 &0\,0\,0\,0\,0 & $A_1 \oplus   A_5$ & 2\\
& 0&\\
&0&\\
3 &0\,0\,1\,0\,0 & $A_2 \oplus   A_2\oplus  A_2$\\
\hline
 &1&\\
&0&\\
4 & 0\,0\,1\,0\,0 & $A_1 \oplus   A_2 \oplus   A_2 \oplus  T_1$ &2\\
 &1&\\
&    0 &\\
 4& 1 \, 0\, 0\, 1 \, 0 & $A_3 \oplus A_1 \oplus T_2$&2\\
&1&\\
&1&\\
8 & 1\,0\,1\,0\,1 & $A_1 \oplus   A_1 \oplus  T_4$&1\\
 &0&\\
&1&\\
8 & 1\, 1\, 0\, 1\,1 & $A_1 \oplus A_1 \oplus T_4$&1\\
\hline%
 &1&\\
&0&\\
6 &  1\,0\,1\,0\,1 & $A_1 \oplus   A_1 \oplus   A_1 \oplus  T_3$\\
 &1&\\
&1&\\
9 & 1\,1\,0\,1\,1   & $A_1 \oplus   T_5$\\
 &1&\\
&1&\\
12 &  1\,1\,1\,1\,1 & $T_6$\\
\hline
  \end{tabular}
\end{table}
%\vspace{20ex}%fill

%%\vspace{-.5ex}

\begin{remark}
\label{rem:7.1}
 The elements $\sigma(m)=e^{\frac{\pi i\rho^\vee}{m}}$,
where $m$ is a regular number, are liftings of all regular elements of $W$
for all classical simple Lie algebras $\fg$ as well. In the case $m=h$
it is the Coxeter-Kostant element, whose diagram has all labels equal 1.
In addition we have $\sigma(n-1)$ for $\fg=\fsl_n$, whose diagram has one 0
label and all other labels equal 1, and we have $\sigma(n)$ for $\fg=\so_{2n}$,
whose diagram is as follows for $n$ even and odd respectively:
\vspace{-1ex}
\begin{equation*}
  \begin{array}{cccccccccccccccccccccccccccccccc}
       &1&&&&& 1&&&&&& && &&1&&&&&& 1 \\
 1 &0&1&0&\ldots &1&0&1&&&&&&&&1 &0&1&0&\ldots &1&0&1&0
  \end{array}
\end{equation*}
%
%%%%%\vspace{-2ex}
%\newpage
The $d$-th powers of the above elements, where $d$ divide their orders, are
liftings of all regular elements of $W$, up to conjugacy. Removing $1$ at the
extra node, we obtain the characteristic of a nilpotent element of $\so_{2n}$
of semisimple type, corresponding to the partition $(n+1,n-1)$ if $n$
is even, and to the partition $(n,n)$ if $n$ is odd.
\end{remark}

\vspace{5ex}

\begin{table}[!h]%[!t]
  \centering\label{tab:6.2}
Table 6.2.  Diagrams of regular elements of $W_{E_7}$\\[1ex]
  \begin{tabular}{c|c|c|c}
    \hline
order m&   $E^{(1)}_7$  & fixed point set & $\dim \fh^w$\\
\hline  %\vspace{.75ex}
&1&\\
2 &0\,0\,0\,0\,0\,0\,0 & $A_7$\\
&0&\\
3 &0\,0\,0\,0\,1\,0\,0 & $A_2 \oplus   A_5$ & 1\\
\hline
&0&\\
7 & 0\,1\,0\,1\,0\,0\,1 & $A_1 \oplus   A_1 \oplus  A_1\oplus
A_2 \oplus  T_2$ &1\\
&0&\\
9 & 0\,1\,0\,1\,0\,1\,1 & $A_1 \oplus   A_1\oplus  A_1\oplus  A_1
\oplus  T_3$ & 1\\
& 1 &\\
9 & 1\, 1\,0\,0\,1\,0\,1 & $A_2 \oplus A_1 \oplus T_4$&1\\
\hline
&0&\\
6 &  1\,0\,0\,1\,0\,0\,1 & $A_1 \oplus   A_2 \oplus   A_2 \oplus  T_2$\\
 &1&\\
14 & 1\,1\,1\,0\,1\,1\,1\,   & $A_1 \oplus   T_6$\\
&1&\\
18 &  1\,1\,1\,1\,1\,1\,1 & $T_7$\\
  \end{tabular}
\end{table}

%\newpage

%
\begin{table}[!h]
\centering
Table 6.3.  Diagrams of regular elements of $W_{E_8}$ \\[1ex]
\label{tab:6.3}
  \centering
  \begin{tabular}{c|c|c}
    \hline
order m&   $E^{(1)}_8$  & fixed point set\\
\hline  %\vspace{.75ex}
&\qquad 0&\\
2 &0\,0\,0\,0\,0\,0\,0\,1 & $D_8$\\
&\qquad 1&\\
3 &0\,0\,0\,0\,0\,0\,0\,0 & $A_8$\\
%%%\hline
&\qquad 0&\\
4 & 0\,0\,0\,1\,0\,0\,0\,0 & $A_3 \oplus  D_5$\\
&\qquad 0&\\
5 & 0\,0\,0\,0\,1\,0\,0\,0 & $A_4 \oplus   A_4$\\
&\qquad 0&\\
8 & 0\,1\,0\,0\,0\,1\,0\,0 & $A_1 \oplus  A_1 \oplus  A_2 \oplus  A_3 \oplus  T_1$\\
\hline
&\qquad 0&\\
6 &  1\,0\,0\,0\,1\,0\,0\,0 & $A_3 \oplus   A_4  \oplus  T_1$\\
 &\qquad 0&\\
10 & 1\,0\,1\,0\,0\,1\,0\,0   & $A_1\oplus  A_1\oplus  A_2\oplus  A_2\oplus  T_2$\\
&\qquad 0&\\
12 &  1\,0\,1\,0\,0\,1\,0\,1 & $A_1\oplus  A_1\oplus  A_1\oplus  A_2\oplus    T_3$\\
&\qquad 0&\\
15 &  1\,1\,0\,1\,0\,1\,0\,1 & $A_1\oplus  A_1\oplus  A_1\oplus  A_1\oplus    T_4$\\
&\qquad 1&\\
20 &  1\,1\,1\,0\,1\,0\,1\,1 & $A_1\oplus  A_1\oplus    T_6$\\
&\qquad 1&\\
24 &  1\,1\,1\,1\,1\,0\,1\,1 & $A_1\oplus   T_7$\\
&\qquad 1&\\
30 &  1\,1\,1\,1\,1\,1\,1\,1 & $T_8$ \\
\hline
  \end{tabular}
\end{table}

\newpage

\vspace*{4ex}

\begin{table}[!h]
Table 6.4.  Diagrams of regular elements of $W_{F_4}$\\[1ex]
\label{tab:6.4}
 %%   \hline\\
  \centering
  \begin{tabular}{c|c|c}
    \hline
order m&   $F^{(1)}_4$\,\,$ \circ-\circ -\circ\Rightarrow \circ-\circ$  & fixed point set\\
\hline  %\vspace{.75ex}
2 &0\,1\,0\,0\,0 & $A_1\oplus  C_3$\\
3 &0\,0\,1\,0\,0 & $A_2 \oplus  A_2$\\
\hline
4 & 1\,0\,1\,0\,0 & $A_1 \oplus  A_2 \oplus  T_1$\\
6 & 1\,0\,1\,0\,1 & $A_1 \oplus   A_1 \oplus  T_2$\\
8 & 1\,1\,1\,0\,1 & $A_1 \oplus  T_3$\\
12 & 1\,1\,1\,1\,1 & $T_4$\\
\hline
  \end{tabular}
\end{table}
% \vspace{20ex}

\vspace*{4ex}
%%%\newpage

\begin{table}%[!t]
  \centering
Table 6.5.  Diagrams of regular elements of $W_{G_2}$\\[1ex]
\label{tab:6.5}
   %% \hline\\
  \begin{tabular}{c|c|c}
    \hline
order m&   $G^{(1)}_2$\,\,$ \circ - \circ\Rrightarrow \circ $  & fixed point set\\
\hline  %\vspace{.75ex}
2 &0\,1\,0 & $A_1\oplus  A_1$\\
\hline
3 &1\,1\,0 & $A_1 \oplus  T_1$\\
6 & 1\,1\,1 & $T_2$\\
\hline
  \end{tabular}
\end{table}

\newpage

\begin{remark}
\label{rem:7.2}

Let $e$ be a nilpotent element of $\fg$ of semisimple type.  Then
we can canonically associate with $e$ (rather its conjugacy class) 
a conjugacy class in $W$ as
follows.  Let $e+F$ be a generic cyclic element, associated with
$e$, which, by definition, is semisimple.  As has been explained
in the Introduction, we associate with $e$
a finite order inner automorphism $\sigma_e$ of $\fg$, with the diagram being
the extended Dynkin diagram whose labels $s_i$ on the Dynkin
subdiagram are the labels of the characteristic of $e$
(resp. the halves of these labels) if $e$ is an odd (resp. even)
nilpotent, and the label $s_0$ of the extra node is $2$
(resp.~$1$).  Let $\fa$ be the semisimple part of the centralizer
$\fz(e+F)$ of $e+F$ in $\fg$.  Since $e+F$ is an eigenvector
of $\sigma_e$, the reductive subalgebra $\fz(e+F)$, and
hence~$\fa$, is $\sigma_e$-invariant.  Let $\fh_0$ be a
Cartan subalgebra of the fixed point set of $\sigma_e$ in $\fa$
and let $\fh'$ be the centralizer of $\fh_0$ in $\fz(e+F)$.  Then by
\cite{Ka2}, Lemma~8.1b, $\fh'$ is a Cartan subalgebra of
$\fz(e+F)$, hence of $\fg$. Therefore $\sigma_e$ induces an element 
$w_e$ of $W(\fh')$, which
fixes $\fh_0$ pointwise.  Hence $w_e$ is independent of the
choice of $\fh_0$.  We thus obtain a well defined map
from the set of conjugacy classes of nilpotent elements of
semisimple type in $\fg$ to the set of conjugacy classes in
$W$.  
%%Recall that the order of $\sigma_e$ equals
% \begin{equation*}
%% $  |\sigma_e | = s_0 + \sum^r_{i=1} a_is_i$
% \end{equation*}
%
%%and that the order $|w_e|$ of $w_e$
%%equals $|\sigma_e|$ or its half (the
%%latter possibility may happen only if $e$ is odd).
%Note also that the dimension of the fixed point set of $w_e$ on $\fh'$ is
%at most the number of zeros among $s'_is$.
This map can be extended to the conjugacy classes of nilpotent elements
of ``quasi-semisimple'' type as follows. By results of Sections 4 and 5,
for any nilpotent element $e$, which is not of regular semisimple type,
there exists a proper regular reducing subalgebra $\fq$. If $e$ is of 
mixed type, it has decomposition $e=e^s + e^n$, where $e^s\in \fq$
is of semisimple type, and $e^n\in \fz(\fq)$ is a nilpotent element of
lower depth that $e$. Let $w_1\in W$ be an element, corresponding
to $e$, as constructed above. If $e^n$ is of mixed type, let $w_2\in W$
be the element, corresponding to $(e^n)^s$. If $e^n$ is of semisimple type
in $\fg$, we let $w=w_1w_2\in W$ correspond to $e$. Otherwise, we repeat this
procedure for the element $(e^n)^n$, etc. If this sequence of steps brings us 
to a nilpotent element of semisimple type, we call $e$ a nilpotent element of
quasi-semisimple type. Thus, to any nilpotent conjugacy class of 
quasi-semisimple type in $\fg$ we have associated a conjugacy class in $W$.
%Most likely this map coincides with 
This map is probably closely related to that, constructed in 
\cite{KL}, \cite{L}, on the set of all nilpotent classes.
% of quasi-semisimple type (but it is more explicit).
Note that in $\fsl_n$ and $\fs\fp_n$ all nilpotent classes are of 
quasi-semisimple type, that nilpotent classes of nilpotent type are not of 
quasi-semisimple type, and only they are not of quasi-semisimple type in all
exceptional $\fg$, with the exception of $E_8$, where there are four
more such classes.
\end{remark}

%\begin{remark}
%  \label{rem:7.3}
%It follows from Remark~\ref{rem:7.2} that for an odd nilpotent
%element $e$, not of nilpotent type, the number $|\sigma_e|$ is even.
%\end{remark}

\begin{example}
  \label{ex:7.3}
Let $e=e_{-\alpha_0}$ be the root vector attached to the highest root
$-\alpha_0$ of a simple Lie algebra $\fg$.  Its orbit has minimal
dimension among all non-zero nilpotent orbits.  The depth of the
corresponding $\ZZ$-grading (\ref{eq:0.1}) is $2$, $\fg_{\pm 2}$
being $\FF e_{\mp \alpha_0}$.  In all cases $e$ is a nilpotent element of
semisimple type.  If $\fg \neq \fs\fp_{2n}$, then $e$ is odd and all
the non-zero labels of the diagram of $\sigma_e$ are: $2$ at the
extra node and $1$ at the adjacent to it nodes; we have:  $|\sigma_e|=4$
and $w_e$ is the reflection with respect to $\alpha_0$.  If $\fg =
\fs\fp_n$, then $e$ is even and all the non-zero labels of the
diagram of $\sigma_e$ are $1$ at the two end nodes; we have:
$|\sigma_e | =2$ and $w_e$ is again the reflection with respect
to $\alpha_0$.

\end{example}

\begin{definition}  \label{def:7.4}

We call an element $w \in W$ of order~$m$ quasiregular if for
any proper divisor~$s$ of~$m$ the element $ w^s$ does not fix a root.

\end{definition}

\begin{lemma}  \label{lem:7.5}

If $w \in W$ is regular, then it is quasiregular.

\end{lemma}

%%%%\newpage

\begin{proof}

This follows from the fact, proved in \cite{S}, that the eigenvalue
of a regular eigenvector
of a regular element $w$ of order $m$ is an $m$-th primitive root of unity.

%Let $a \in \fh$ be a regular eigenvector of~$w$.  We can choose a
%set of positive roots, such that~$a$ lies in the open fundamental
%chamber.  Then for any proper divisor $s$ of $m$ we have:
%
%\begin{equation}
%  \label{eq:7.1}
%  w^{-s} (a) = \lambda a  \,\, {\rm for\,\, some} \,\, \lambda \in \CC \backsl%ash \{ 1 \},
%\end{equation}
%
%since otherwise $w^{-s}=1$.
%simple reflection $r_{\alpha}$ in a reduced decomposition of $w^{-s}$
%fixes~$a$, hence $\alpha (a) =0$ and~$a$ is not regular.  But
%But then, if $\beta$ is a root, we get from (\ref{eq:7.1}):  $(w^s
%\beta) (a) = \lambda \beta (a)$, hence $\beta$ is not fixed by
%$w^s$, which proves that $w$ is quasi-regular.
\end{proof}

\begin{theorem}
  \label{th:7.6}
If $w$ is a quasiregular element of order~$m$ in~$W$ and it has a
lift $\sigma$ of the same order in the normalizer of $\fh$ in the
adjoint group, then
\begin{equation}
  \label{eq:7.2}
    \frac{|\Delta |}{m} = \dim \fg^\sigma - \dim \fh^w\, ,
\end{equation}
where $|\Delta |$ is the number of roots of $\fg$ and
$\fg^{\sigma}$ (resp.~$\fh^w$) denotes the fixed point set of
$\sigma$ in~$\fg$ (resp. of $w$ in $\fh$).
\end{theorem}

\begin{proof}
Let $C$ (resp.~$\tilde{C}$) be the cyclic group, generated by $w$
(resp.~$\sigma$).  Then, by Lemma \ref{lem:7.5},
$|C(\beta)| = m = |\tilde{C}(e_{\beta})|$ for
any root $\beta$ and a root vector $e_\beta$, attached to
$\beta$.  Hence $\dim \fg^\sigma = \frac{|\Delta|}{m} + \dim \fh^w$.
\end{proof}

\begin{remark}
  \label{rem:7.7}

By the construction of $w_e$, we have:
\begin{equation*}
  \dim \fh^{w_e} \leq {\rm rank }\, \fr \, ,
\end{equation*}
where $\fr$ is the reductive part of the centralizer of $e$ in
$\fg$.  It follows that $\fh^{w_e} =0$ if $e$ is a distinguished
nilpotent of $\fg$.  Using formula (\ref{eq:7.2}) one can compute
$\dim \fh^{w_e}$ for any regular~$w\in W$.
%A quick inspection of Tables~4.1--4.5 shows that $\fh^w =0$ in
%all cases except for the following ones:
%
%\begin{equation*}
%  \begin{array}{llll}
%    \fg = E_6 \, , \, |w| = 2,4 &:
%         &\dim \fh^w=2\, ,\,\, {\rm or }\,\, |w| =8 \, :\, \dim \fh^w=1\, ;\\
%
%         \fg = E_7 \, , \, |w| = 3,7,9 &:
%         & \dim \fh^w =1 \, .
%      %
%  \end{array}
%\end{equation*}
%

\end{remark}

\begin{remark}
  \label{rem:7.8}

%As shown in \cite{K2}, all lifts of an element $w \in W$ are
%conjugate if and only if $\fh^w=0$.  It turns out that even the
%lifts of regular elements~$w$ of~$W$, obtained using cyclic
%elements, can be non-conjugate if $\fh^w\neq 0$.  Such lifts can
In the cases when a regular $w\in W$ has two liftings, the second one in
Tables 6.1 and 6.2 is obtained using good gradings for~$e$, different from
Dynkin gradings, found in \cite{EK}.
%Among other elements of exceptional Weyl groups there
%are two such examples:
%
%\begin{equation*}
%  \begin{array}{ccccc}
%    &&0\\
%    &&1\\
%  1&1&0&1&1
%  \end{array} \hbox{\quad and \quad}
%
%  \begin{array}{ccccccc}
%&&&\\
%  &&&1\\
%  1&1&0&0&1&0&1 \, ,
%  \end{array}
%\end{equation*}
%
%corresponding to regular elements of order~$8$ and~$9$ of the
%Weyl groups of type~$E_6$ and~$E_7$, respectively.
One of the $E_6$ examples was found in \cite{P2}.

The group $S$ of symmetries of the Dynkin diagram of~$\fg$ acts
on~$\fh$ by permuting simple roots, and we may consider the
twisted Weyl group $W^{tw} = S \ltimes W$.  In \cite{S} Springer
also studied  the regular elements of $W^{tw}\backslash W$,
called twisted regular, and
showed that again there is at most one such twisted regular
element of given order and found these orders.  Lifts of twisted
regular elements are outer automorphisms of finite order of $\fg$,
hence they are
described by labeled twisted extended Dynkin diagrams \cite{Ka2}, \cite{OV}.

Given a symmetry $\nu$ of the Dynkin diagram of a simple Lie
algebra, let $\fp$ be its fixed point set.  Then the
corresponding twisted extended Dynkin diagram is obtained by
adding to the Dynkin diagram of $\fp$ one extra node.  If now $e$
is an even nilpotent elements of $\fp$, we associate to it a
labeling of the twisted extended Dynkin diagram by putting halves
of the labels of the characteristic of $e$ in $\fp$ at the nodes
of the Dynkin diagram of $\fp$ and $1$ at the extra node.  If $e$
is of regular semisimple type in $\fp$, we obtain a finite order
automorphism $\sigma_e$ of $\fg$ and the corresponding element
$w_e$ of the same order of $W^{tw}\backslash W$ acting on the
centralizer of a generic cyclic element $e+F$.

\end{remark}

In the following Tables 6.6 and 6.7 we list in the lower part
the diagrams of outer finite order automorphisms of $E_6$ and $D_4$
from the connected component of $\nu$ of order 2 and 3 respectively,
which induce the regular elements of
$W^{tw}\backslash W$, corresponding to distinguished nilpotent
elements of $\fp$, and in the upper part the lifts of the
remaining twisted regular elements.

\begin{center}%[t]
Table 6.6  Diagrams of regular elements of
$W^{tw}_{E_6}$\\[1ex]
\label{tab:6.6}
 %%   \hline\\
  \centering
  \begin{tabular}{c|c|c|c}
    \hline
order m&   $E^{(2)}_6$\,\,$ \circ-\circ -\circ\Leftarrow
\circ-\circ$  & fixed point set & $\dim \fh^w$\\
\hline  %\vspace{.75ex}
2 &0\,0\,0\,0\,1 & $C_4$\\
4 & 0\,0\,0\,1\,0 & $A_3 \oplus  A_1$\\
8 & 1\,0\,0\,1\,1 & $A_2 +T_2$ &1\\
8 & 0\,1\,0\,1\,0 & $A_1 \oplus A_1 \oplus A_1 \oplus  T_1$ &1\\
\hline
6 & 1\,0\, 0\,1\,0 & $A_2 \oplus   A_1 \oplus  T_1$\\
12 & 1\,1\,0 \,1\,1 & $A_1 \oplus  T_3$\\
18 & 1\,1\,1\,1\,1 & $T_4$\\

  \end{tabular}
\end{center}

\vspace*{4ex}

\begin{center}%[t]
Table 6.7.  Diagrams of regular elements of $W^{tw}_{D_4}$\\[1ex]
\label{tab:6.7}
 %%   \hline\\
  \centering
  \begin{tabular}{c|c|c}
    \hline
order m&   $D^{(3)}_4$\,\,$ \circ-\circ\Lleftarrow \circ$  & fixed point set\\
\hline  %\vspace{.75ex}
3 &0\,0\,1 & $A_2$\\
\hline
6 & 1\,0\,1 & $  A_1 \oplus  T_1$\\
12 & 1\,1\,1 & $T_2$\\
\hline
  \end{tabular}
\end{center}

\vspace*{4ex}

\begin{remark}
\label{rem:7.9}
Theorem~\ref{th:7.6} holds also for twisted regular elements. It follows
that for all cases listed in Tables 6.6
and 6.7 we have $\fh^w=0$, except for the element $w$ of order 8 in
$W^{tw}_{E_6}$, when
$\dim \fh^w=1$.
The latter element is the only twisted regular element which is not a power
of an element of the form $\sigma_e$, but the
diagrams of its liftings can be easily obtained using formula (\ref{eq:7.2})
(one can see that $\dim \fh^w=1$ from Table~8 of \cite{S}).
It turns out that again, all automorphisms are of the form
$\nu e^{\frac{\pi i\rho^\vee}{m}}$,
%e^{\frac{\pi i \rho^\vee}{m}}$,
where $m$ is from the first column,
occur in these tables (in the case $m=8$ it is the first one).
\end{remark}

\begin{proposition}
  \label{prop:6.11}
Let $G^0|\fg^1$ be a theta group, attached to an inner finite order
automorphism $\sigma$ of $\fg$, such that there exists $v \in
\fg^1$ whose $G^0$-orbit is closed and the stabilizer $G^0_v$ is
finite.  Then the centralizer $\fh: = \fg_v$ of $v$ in $\fg$ is a
Cartan subalgebra of $\fg$, hence $\sigma$ induces an element $w$
of the Weyl group $W$, attached to $\fh$.  Moreover, $w$ is a
regular element of $W$ and $\fh^w=0$.  Conversely, all regular
elements of $W$ with zero fixed point set on $\fh$ are obtained
in this way.

\end{proposition}

\begin{proof}
Since the $G^0$-orbit of $v$ is closed, $v$ is a semisimple
element of $\fg$ \cite{V}, hence its centralizer $\fg_v$ is
reductive.  The automorphism $\sigma$ induces a finite order
automorphism of $\fg_v$ with zero fixed point set.  Hence $\fh
: =\fg_v$ is a Cartan subalgebra \cite{Ka1}, Chapter~8, and $v$ is a
regular element of $\fg$, so $\sigma$ induces a regular Weyl
group element $w$ on $\fh$.  We have $\fh^w=0$ since the fixed
point set of $\sigma$ on $\fg_v$ is zero.

Conversely, if $w$  is a regular element of $W$, with a regular
eigenvector $v \in \fh$ and such that $\fh^w =0$, then its
(unique, up to conjugacy) lift $\sigma$ in $G$ obviously defines
a theta group $G^0|\fg^1$ for which $v$ has a closed $G^0$-orbit
with a finite stabilizer.
\end{proof}

\begin{remark}
\label{rem:7.12}
% For $\fg =E_6, E_7, E_8$ all
All theta groups with properties, described in
Proposition \ref{prop:6.11},
have been independently found in a recent paper \cite{G} on classification of
theta groups of positive rank.
Their tables are in complete agreement
with our Tables 5.1--5.4 and 6.6, 6.7, as predicted by Proposition
\ref{prop:6.11}.

% with $\fh^w=0$, have an element, whose orbit is closed and the stabilizer
% is finite (hence such elements form a non-empty Zariski open subset).
% Remarkably, it follows from the classification in [] that these are all
% $\theta$-groups with this property

\end{remark}
%\newpage

%%%%%%%%%%%%%%%%%%%%%%%%%%%%%%%%%%%%%%%%%%%%%%%%%%%%%%%%%%%%%%%%%%%%%%%%%%%%%%%%%%%%%%%%%

\end{document}